\documentclass{amsart}

\usepackage{graphicx}
\usepackage[utf8]{inputenc}
\usepackage[english]{babel}

\usepackage{amsmath,amssymb,amsthm,enumerate,url,appendix,tabularx,stmaryrd,mathrsfs,color,mathtools,MnSymbol}
\usepackage{afterpage}
\usepackage{comment}
\usepackage{caption}

\usepackage[labelformat=simple]{subcaption}

\usepackage{graphicx}
\usepackage{epstopdf}

\renewcommand{\Re}{\operatorname{Re}}

\usepackage[margin=2cm]{geometry}

\theoremstyle{plain}
\newtheorem{theorem}{Theorem}[section]
\newtheorem{lemma}[theorem]{Lemma}
\newtheorem{corollary}[theorem]{Corollary}
\newtheorem{proposition}[theorem]{Proposition}

\newtheorem{assumption}[theorem]{Assumption}
\newtheorem{problem}[theorem]{Problem}

\newtheorem{definition}[theorem]{Definition}
\newtheorem{example}[theorem]{Example}

\newtheorem{remark}[theorem]{Remark}

\numberwithin{equation}{section}
\numberwithin{figure}{section}
\numberwithin{table}{section}

\newcommand{\R}{\mathbb{R}}
\newcommand{\C}{\mathbb{C}}
\newcommand{\N}{\mathbb{N}}

\newcommand{\bigO}{\mathcal{O}}

\newcommand{\microspace}{\mspace{0.5mu}}

\newcommand{\abs}[1]{\left|#1\right|}
\newcommand{\norm}[1]{\left\Vert#1\right\Vert}
\newcommand{\triplenorm}[1]{
  \left\lvert\!\microspace\left\lvert\!\microspace\left\lvert #1 
        \right\rvert\!\microspace\right\rvert\!\microspace\right\rvert_{\Gamma}}

\newcommand{\dualproduct}[3][\Gamma]{\left<#2,#3 \right>_{#1}}
\newcommand{\ltwoproduct}[3][\R^d \setminus \Gamma]{\left(#2,#3\right)_{L^2\left(#1\right)}}
\newcommand{\ltwonorm}[2][\R^d \setminus \Gamma]{\norm{#2}_{L^2\left(#1\right)}}
\newcommand{\colvec}[1]{\begin{pmatrix}#1\end{pmatrix}}
\newcommand{\innerproduct}[3][L^2\left(\R^d\right)]{\left(#2,#3\right)_{#1}}

\newcommand{\domain}{\operatorname{dom}}
\newcommand{\laplace}{\Delta}

\def\TT{\mathcal{T}}

\def\ii{i}
\def\AA{\mathcal{A}}

\newcommand{\usg}{u_{sg}}
\newcommand{\vsg}{v_{sg}}

\newcommand{\uie}{u_{ie}}
\newcommand{\vie}{v_{ie}}
\newcommand{\uiet}{\widetilde{u}_{ie}}
\newcommand{\viet}{\widetilde{v}_{ie}}
\newcommand{\usgh}{{u}_{sg,h}}
\newcommand{\vsgh}{{v}_{sg,h}}

\newcommand{\dd}{\partial_t^{\Delta t}} 

\newcommand{\HH}{\mathcal{H}_h}
\newcommand{\Hpglobal}[1]{H^{#1}\left(\R^d \setminus \Gamma\right)}
\newcommand{\HpLglobal}[1]{H^{#1}_{\laplace}\left(\R^d \setminus \Gamma\right)}
\newcommand{\ltwoglobal}{L^{2}\left(\R^d \setminus \Gamma\right)}
\newcommand{\Hpnorm}[3][\R^d \setminus \Gamma]{\norm{#3}_{H^{#2}\left(#1\right)}}
\newcommand{\tracejump}[1]{{\left\llbracket \gamma #1 \right \rrbracket}}
\newcommand{\normaljump}[1]{{\left\llbracket \partial_n #1 \right \rrbracket}}

\newcommand{\tracemean}[1]{\left\{\!\!\left\{ \gamma #1 \right\}\!\! \right\}}
\newcommand{\normalmean}[1]{\left\{\!\!\left\{ \partial_n #1 \right\}\!\! \right\}}

\newcommand{\bdryinterpY}{J^{Y_h}_{\Gamma}}
\newcommand{\volumeinterpY}{J^{Y_h}_{\Omega^-}}

\newcommand{\eex}{\hbox{}\hfill\rule{0.8ex}{0.8ex}}
\newcommand{\eremk}{\eex}

\iffalse
\usepackage{tikz}
\usepackage{pgfplots}
\usetikzlibrary{pgfplots.groupplots}

\newcommand{\includeTikzOrEps}[1]{\tikzexternalenable 
  \tikzsetnextfilename{#1_img}
  {\include{#1}} \tikzexternaldisable} 
\usetikzlibrary{external}
\tikzexternaldisable
 
\else
\newcommand{\includeTikzOrEps}[1]{\includegraphics{#1_img}}
\fi

\title{Convolution quadrature for the wave equation with a nonlinear impedance boundary condition}
\author{Lehel Banjai}
\address{Maxwell Institute for Mathematical Sciences, School of Mathematical \& Computer Sciences \\
  Heriot-Watt University \\Edinburgh EH14 4AS, UK}
\email{l.banjai@hw.ac.uk}
\author{ Alexander Rieder }
\address{Institute for Analysis and Scientific Computing (Inst. E 101), Vienna University of
Technology\\ Wiedner Hauptstraße 8-10\\ 1040 Wien, Austria}
\email{alexander.rieder@tuwien.ac.at}

\thanks{ Financial support by the Austrian Science Fund (FWF) through the 
doctoral school ``Dissipation and Dispersion in Nonlinear PDEs'' (project W1245, A.R.).}

\date{\today}


\begin{document}

\begin{abstract}
A rarely exploited advantage of time-domain boundary integral equations compared
to their frequency counterparts is that they can be used to treat certain
nonlinear problems. In this work we investigate the scattering of acoustic
waves by a bounded obstacle with a nonlinear impedance boundary condition. We
describe a boundary integral formulation of the problem and prove without any smoothness assumptions on the solution the convergence of a full discretization: Galerkin in space and convolution quadrature in time. If the solution is sufficiently regular, we prove that the discrete method converges at optimal rates. Numerical evidence in 3D supports the theory.
\end{abstract}
\maketitle

\section{Introduction}
\label{sect:introduction}
We propose and analyse a discretization scheme for the linear wave equation subject to a nonlinear boundary condition. The scheme is based on a boundary element method in space and convolution quadrature in time, using either an implicit Euler or BDF2 scheme for its underlying time-discretization. The motivation for the nonlinear boundary condition comes from nonlinear acoustic boundary conditions as investigated in \cite{graber} and from boundary conditions in electromagnetism obtained by asymptotic approximations of thin layers of nonlinear materials \cite{HadJ}. Another source of interesting nonlinear boundary conditions is the coupling with nonlinear circuits \cite{1266883}. Compared with these references, the nonlinear boundary condition that we use is simple. Nevertheless, to the best of our knowledge there are currently no works in the literature analysing the use of time-domain boundary integral equations for nonlinear problems and the nonlinear condition we consider is sufficiently interesting to require a new theory upon which the analysis of more involved applications can be built.

The case of linear boundary conditions has gathered considerable interest in the recent years, and can be considered well understood \cite{AbbJRT,BamH,Bamh2,bls_fembem,BanLM,MR3273900,MR2875249,laliena_sayas,MR3119712}; see in particular the recent book \cite{sayas_book}.
Much of the analysis available in the literature, starting with the groundbreaking work of Bamberger and Ha Duong \cite{BamH}, is based on estimates in the Laplace domain. In the nonlinear case these are not available and the regularity of the solutions is not well understood.
In order to deal with these difficulties we develop an alternative approach  based on an equivalent formulation as a partial differential equation posed in exotic Hilbert spaces. Structurally, these problems are similar and are inspired by the exotic transmission problems of
Laliena and Sayas \cite{laliena_sayas} formulated in the Laplace domain  --- by taking the Z-transform they indeed become equivalent. A similar approach has recently been  used to investigate the coupling of finite elements and convolution quadrature based boundary elements for the Schrödinger equation in \cite{schroedinger}. The focus on analysing the convolution quadrature scheme in the time domain is also present in \cite{banjai_laliena_sayas_kirchhoff_formulas} and \cite{dominguez_sayas_2013}. This reformulation allows us to investigate stability and convergence using the tools from nonlinear semigroup theory. Due to the difficulty regarding the regularity of the exact solution, most of the paper focuses on showing unconditional convergence for low regularity solutions. In Section~\ref{subsect:high_regularity} we then also give a theorem which guarantees the full convergence rate if the exact solution  possesses sufficient regularity.  The paper concludes with numerical experiments in 3D that support and supplement the theoretical results.

\section{Model problem and notation}
\label{sect:model_problem}
We consider the wave equation with a nonlinear impedance boundary condition.
Let $\Omega^- \subseteq \R^d$ be a bounded Lipschitz domain and denote 
the exterior by $\Omega^+:=\R^d \setminus \overline{\Omega^{-}}$ and the boundary by $\Gamma:=\partial \Omega^-$.
The respective trace operators are denoted by $\gamma^{\pm}$ and the normal
derivatives by $\partial^{\pm}_n$, where the normal vector $n$
is taken in both cases as pointing out of $\Omega^-$. We define jumps and mean values as
\begin{align*}
  \tracejump{u}&:=\gamma^+ u - \gamma^- u, & \normaljump{u}&:=\partial_n^+ u -  \partial_n^- u. \\
  \tracemean{u}&:=\frac{1}{2}\left(\gamma^+ u + \gamma^- u\right),
   & \normalmean{u}&:=\frac{1}{2}\left(\partial_n^+u +  \partial_n^- u \right).
\end{align*}

Given a function $g:\R \to \R$, we consider the following model problem:
\begin{align}
\label{eq:wave_eqn}
  \frac{1}{c^2} \ddot{u}^{tot} &= \laplace u^{tot}, \quad \text{ in } \Omega^+, \\
  \partial_n^+ u^{tot} &= g(\dot{u}^{tot}), \quad \text{ on } \Gamma,
\end{align}
together with the initial condition $u^{tot}(t)=u^{inc}(t)$ for all $t\leq 0$,
where $u^{inc}(t)$ is the incident wave satisfying the wave equation
\begin{align}
  \label{eq:wave_eqn_uinc}
  \frac{1}{c^2} \ddot{u}^{inc}(x,t) &= \laplace u^{inc}(x,t), \quad \forall (x,t) \in \Omega^+ \times \R.
\end{align}
We assume that at time $t = 0$ the incident wave has not reached the scatterer and hence  $u^{inc}(x,t)$ vanishes in a neighborhood of $\Omega^-$ for $t\leq 0$.
We further set $u^{inc}(x,t) \equiv 0$ in $\Omega^-$, $\forall \; t \in \R$.

\begin{remark}
  The definition $u^{inc}(x,t) \equiv 0$ in $\Omega^-$ is somewhat uncommon, but helps simplify
  later calculations.  \eremk
\end{remark}

We will make use of a number of standard function spaces.
We start with the space of all smooth test functions with compact support on an open set $\mathcal{O}$,
which will be denoted by $C_{0}^{\infty}(\mathcal{O})$. 
The usual Lebesgue and Sobolev spaces on
a set $\mathcal{O}$ which is either open in $\R^d$ or a relatively open subset of $\Gamma$
will be denoted by $L^p(\mathcal{O})$ and $H^{s}(\mathcal{O})$ respectively.
By $\dualproduct{u}{v}$ we mean the continuous extension of the (complex) $L^2$-product on $\Gamma$
to $H^{-1/2}(\Gamma) \times H^{1/2}(\Gamma)$, i.e. $\dualproduct{u}{v}:=\int_{\Gamma}{u\overline{v}}$ for $u,v \in L^2(\Gamma)$.
We also define the space
$\HpLglobal{1}:=\left\{u \in \Hpglobal{1}: \laplace u \in L^2(\R^d \setminus \Gamma)\right\}$,
where the Laplacian is meant in the sense of distributions
for test functions in $C_{0}^{\infty}(\R^d \setminus \Gamma)$,
i.e.\ taken separately in $\Omega^+ $ and $\Omega^-$. The norm on this space is given by
$\norm{u}^2_{\HpLglobal{1}}:=\norm{u}^2_{\Hpglobal{1}} + \ltwonorm{\laplace u}^2$.
It is common that estimates depend on some generic constants, therefore we use
the notation $A\lesssim B$ to mean that there exists a constant $C>0$ independent of the main
quantities of interest like time or space discretization parameter, such that $A \leq C B$. 
We write $A \sim B$ for $A \lesssim B$ and $B \lesssim A$.

\begin{assumption}
\label{assumption:nonlinearity_g}
We will make the following assumptions on $g$:
\begin{enumerate}[(i)]
  \item $g \in C^1(\R)$,  \label{assumption:nonlinearity_g_c1}
  \item $g(0)=0$,
  \item $g(\mu) \mu \geq 0$, $\forall \mu \in \R$,  \label{assumption:nonlinearity_gss}
  \item $g'(\mu) \geq 0$, $\forall \mu \in \R$,  \label{assumption:nonlinearity_g_prime}
  \item
    \label{assumption:nonlinearity_g_growth}
    $g$ satisfies the growth condition $\abs{g(\mu)} \leq C (1+ \abs{\mu}^p)$, where
    \begin{align*}
      \begin{cases}
        1<p<\infty & d=2, \\
        1<p\leq \frac{d}{d-2} & d \geq 3.
      \end{cases}
    \end{align*} 
\item \label{assumption:g_is_striclty_monotone} $g$ is strictly monotone, i.e.\ there exists $\beta > 0$ such
  that
  \begin{align}
    \label{def:strictly_monotone}
    \left(g(\lambda) - g(\mu) \right)(\lambda-\mu) \geq \beta \abs{\lambda-\mu}^2, \quad \forall \lambda,\mu \in \R.
  \end{align}     
\end{enumerate}
\end{assumption}

\begin{remark}
  The growth condition is such that the operator $\eta \mapsto g(\eta)$ becomes a bounded operator, 
  i.e.\ we have the estimate $\norm{g(u)}_{H^{-1/2}(\Omega)} \leq C (1+\norm{u}_{H^{1/2}(\Gamma)}^p)$ as will be proved  in Lemma~\ref{lemma:g_is_bounded}. \eremk
\end{remark}

\begin{remark}
  Under the conditions posed on $g$, equation~\eqref{eq:wave_eqn} is well-posed.
  This has been shown with more general boundary conditions in
  \cite{lasiecka_tataru} and \cite{graber}, but with slightly stricter growth 
  conditions on $g$. The well-posedness of the stated problem will follow as a special case
  of Theorem~\ref{thm:semigroup}. \eremk
\end{remark}

\begin{remark}
  Assumption (\ref{assumption:g_is_striclty_monotone}) is needed to obtain explicit error bounds for the spatial discretization.
  As it may be an overly restrictive condition in some cases of interest, in Section~\ref{sect:general_g} we sketch what happens if this assumption is dropped.
\end{remark}

\section{Boundary integral equations and discretization}
In order to discretize the problem, it is more convenient to work with homogeneous initial 
conditions $u(0)=\dot{u}(0)=0$. Therefore, we make the decomposition ansatz $u^{tot}=u^{inc} + u^{scat}$. Since $u^{inc}$ satisfies the wave equation it follows that  $u^{scat}$ satisfies
the wave equation with a homogeneous initial condition, i.e.
\begin{align}
  \label{eq:wave_eqn_scattered}
  \frac{1}{c^2} \ddot{u}^{scat} &= \laplace u^{scat}, \quad \text{ in } \Omega^+ \\
  \partial_n^+ u^{scat} &= g(\dot{u}^{scat} + \dot{u}^{inc}) - \partial_n^+ u^{inc}, \quad \text{ on } \Gamma\\
  u^{scat}(t)&=0, \quad \text{ in } \R^d \text{ for all $t \leq 0$}.
\end{align}

For the rest of the paper, we assume $c=1$ to simplify the notation.
In order to reformulate the differential equation in terms of integral equations on $\Gamma$, 
we will need the following integral operators, the properties of which can be found in most 
books on boundary element methods, e.g.
\cite{book_sauter_schwab, book_steinbach, book_mclean,book_hsiao_wendland}.
\begin{definition}
\label{def:integral_operators}
  For $s \in \C^+:=\left\{s \in \C: \Re(s) > 0 \right\}$,
  the Green function associated with
  the differential operator $\Delta - s^2$ is given by:
  \begin{align*}
    \Phi(z;s):=\begin{cases}
      \frac{\ii}{4} H_0^{(1)}\left(\ii s \abs{z}\right), & \text{ for }d=2, \\
      \frac{e^{-s\abs{z}}}{4 \pi \abs{z}}, & \text{ for } d = 3,
    \end{cases}
  \end{align*}
  where $H_0^{(1)}$ denotes the Hankel function of the first kind and order zero.
  We define the single- and double-layer potentials:
  \begin{align*}
    \left(S(s) \varphi \right)\left(x\right):=\int_{\Gamma}{\Phi(x-y;s) \varphi(y) \;dy} \\
    \left(D(s) \psi \right)\left(x\right):=\int_{\Gamma}{\partial_{n(y)}\Phi(x-y;s) \psi(y) \;dy}.
  \end{align*}

  For all $u\in \HpLglobal{1}$, with $\laplace u - s^2 u =0$, the representation formula 
\begin{align*}
    u(x)=-S(s) \normaljump{u}(x) + D(s) \tracejump{u} (x)
  \end{align*}
holds.
  
  Finally, we define the corresponding boundary integral operators:
\begin{subequations}
  \label{eq:mapping_properties}
  \begin{align}
    V(s): \;\;  &H^{-1/2}(\Gamma)\to H^{1/2}(\Gamma), &
                                V(s)&:=\gamma^{\pm} S(s), \\
    K(s): \;\;  &H^{1/2}(\Gamma)\to H^{1/2}(\Gamma),  &
                             K(s)&:=\tracemean{D(s)}, \\
    K^t(s): \;\; & H^{-1/2}(\Gamma)\to H^{-1/2}(\Gamma), & 
                               K^t(s)&:=\normalmean{S(s)},\\
    W(s):  \;\; &H^{1/2}(\Gamma) \to H^{-1/2}(\Gamma), &
                              W(s)&:=- \partial_n^{\pm} D(s).
  \end{align}
\end{subequations}
\end{definition}

In order to solve the wave equation, we define the Calderón operators
\begin{align}
  \label{eq:def:B}
  B(s)&:=\begin{pmatrix}
    s V(s)  & K(s) \\
    -K^t(s) & s^{-1} W(s)
\end{pmatrix},\\
  B_{\text{imp}}(s)&:=B(s) + \begin{pmatrix}
    0 & -\frac{1}{2}I \\
    \frac{1}{2} I & 0 \end{pmatrix}.
  \label{eq:def:B_imp}
\end{align}

\begin{definition}
\label{def:operational_calculus}
  In this paper, we make use of the operational calculus notation as  
 is common in the literature on convolution quadrature \cite{lubich_94}.
Note that  the corresponding operational calculus dates back much further, see e.g. \cite{mikusinski} and
  \cite{yosida_operational_calculus}.
  Let $K(s): X\to Y$ be a family of bounded linear operators
  analytic for $\Re(s) > 0$, and
  let $\mathscr{L}$ denote the Laplace transform and $\mathscr{L}^{-1}$ its inverse.
  We define
  \begin{align*}
    K(\partial_t)g:= \mathscr{L}^{-1}\big( K(\cdot) \mathscr{L} g \big),
  \end{align*}
  where $g \in \domain\left(K(\partial_t)\right)$ is such, that the 
  inverse Laplace transform exists, and the expression above is well defined.

  This operation has the following important properties:
  \begin{enumerate}[(i)]
    \item For kernels $K_1(s)$ and $K_2(s)$, we have
      $\displaystyle
        K_1(\partial_t) K_2(\partial_t) = \left(K_1 K_2\right)\left(\partial_t \right).
        $
    \item For $K(s):=s$, we have:
      $K(\partial_t) g(t)=g'(t)$, $\forall g\in C^{1}(\R^+)$, with $g(0)=0$.
    \item For $K(s):=s^{-1}$ we have:
      $K(\partial_t) g(t)=\int_{0}^{t}{ g(\xi)\, d\xi}$, $\forall g \in C(\R^+)$.
    \end{enumerate}
    The last point motivates the notation $\partial_t^{-1}$ for the integral, which will
    be important when we introduce a corresponding discrete version.
\end{definition}

Using the definition above, we can easily transfer the representation formula from the
Laplace domain to the time domain to get Kirchoff's representation formula:
If $u\in C^2\left(\R,\HpLglobal{1}\right)$ solves the wave equation in $\R^d \setminus \Gamma$ then it can be written 
as:
\begin{align}
\label{eq:kirchhoff_formula}
  u =-S(\partial_t) \normaljump{u} + D(\partial_t) \tracejump{u}.
\end{align}
This representation formula provides us with the connection between the PDE and the
boundary integral formulation, which we will use for our discretization.
Namely, with
\begin{align}
  \label{eq:continuous_int_eq}
  B_{\text{imp}}(\partial_t) \colvec{\varphi \\ \psi} + \colvec{0 \\ g(\psi + \dot{u}^{inc}) }
&= \colvec{0 \\ -\partial^+_n u^{inc}},
\end{align}
the following equivalence holds.
\begin{enumerate}[(i)]
\item If $u:=u^{scat}$ solves \eqref{eq:wave_eqn_scattered},  then
  $(\varphi,\psi)$, with $\varphi:=-\partial_n^+ u$ and $\psi:=\gamma^+ \dot{u}$,
  solves~\eqref{eq:continuous_int_eq}.
\item If   $(\varphi,\psi)$ solves~\eqref{eq:continuous_int_eq}, then
  $u:=S(\partial_t) \varphi + \partial_t^{-1} D(\partial_t) \psi$ solves
  \eqref{eq:wave_eqn_scattered}.
\end{enumerate}

This statement follows from Kirchhoff's representation formula for the wave equation
(see \cite{banjai_laliena_sayas_kirchhoff_formulas} and the references therein) 
and the definition of the boundary integral operators in Definition~\ref{def:integral_operators}. We will
not go into details here, as we will not directly make use of this result. Instead we will
later prove a discrete analogue in Lemma~\ref{lemma:diff_eq_u_ie}.

We consider two closed sub-spaces $X_h \subseteq H^{-1/2}(\Gamma)$, $Y_h \subseteq H^{1/2}(\Gamma)$ not necessarily
finite dimensional and let $\bdryinterpY: H^{1/2}(\Gamma) \to Y_h$ denote a stable operator with ``good''
approximation properties. This can be the Scott-Zhang operator in its variant based on pure element averaging, see for example \cite[Lemma 3]{aff_hypsing}. An alternative is the $L^2$-projection for low order piecewise polynomials, where the stability depends on the triangulation used with quasiuniformity of the  triangulation
being a sufficient assumption; see 
\cite{crouzeix_thomee_stability_l2projection,bank_yserentant_l2projection} for other sufficient conditions.
The detailed approximation requirements for the projection operator and the discrete spaces can be found in Assumption~\ref{assumption:approximation_spaces}
or Lemma~\ref{lemma:convergence_space_general_g} respectively.

For the rest of the paper, we fix a time step size $\Delta t > 0$ and 
use the abbreviation $t_n:=n\Delta t$. For time discretization we will use the two $A$-stable backward difference formulas BDF1 and BDF2. Applied to  $\dot{u}=f(t,u)$,
with step-size $\Delta t$, these give the recursion
\begin{align*}
  \frac{1}{\Delta t}\sum_{j=0}^{k}{\alpha_j u^{n-j}}&=f(t_n,u^n),
\end{align*}
where $k = 1$ and $\alpha_0 = 1$, $\alpha_1 = -1$ for the one-step BDF1 and $k = 2$ and $\alpha_0 = 1/2$, $\alpha_1 = -2$, $\alpha_2 = 3/2$ for the two-step BDF2 method. Apart from the $A$-stability we will also require the fact that these methods are $G$-stable as shown by Dahlquist \cite{Dah}. In the following $u(t)$ is assumed to be in a Hilbert space with an inner product $\langle \cdot, \cdot \rangle$.

\begin{proposition}\label{prop:Gstability}
  The linear multistep methods BDF1 and BDF2 are $G$-stable. Namely there exists a positive definite matrix $G = (g_{ij})_{i,j = 1,\dots,k}$ such that
\[
\Re \left \langle \sum_{j = 0}^k \alpha_j u^{n-j},u^n \right \rangle 
\geq \|U^n\|^2_G-\|U^{n-1}\|^2_G,
\]
where $U^n = (u^n, \dots, u^{n-k+1})^T$ and 
\[
\|U^n\|_G^2 = \sum_{i = 1}^k\sum_{j = 1}^k g_{ij} \langle u^{n-k+i},u^{n-k+j}\rangle.
\]
\end{proposition}
\begin{proof}
  As BDF methods are equivalent to their corresponding one-leg methods, the result follows from \cite[Chapter V.6, Theorem 6.7]{hairer_wanner_2} and its proof.
\end{proof}

Next, we give the discrete analogue to Definition~\ref{def:operational_calculus}; this is standard in the CQ literature(see\cite{lubich_cq1,lubich_cq2,lubich_94}). 
To do this we require a standard result on multistep methods.

\begin{proposition}[{\cite[Chapter V.1, Theorem 1.5]{hairer_wanner_2}}]
As BDF1 and BDF2 are $A$-stable methods their generating function
  $\displaystyle
    \delta(z):=\sum_{j=0}^{k}\alpha_j z^j$
satisifes  $\Re{\delta(z)} > 0$ for $\abs{z}<1$.
\end{proposition}

\begin{definition}
\label{def:cq_opearational_calculus}
Analogous to the Laplace transform $\mathscr{L}$, we define the $Z$-transform $\mathscr{Z}$ 
of a sequence $g=(g_{n})_{n=0}^{\infty}$ as the power series
$\displaystyle
  \left(\mathscr{Z}g \right):=\sum_{n=0}^{\infty}{ g_n z^n}$.
We will also often use the shorthand $\widehat{g}:=\mathscr{Z}(g)$. 

Let $K(s)$ again be an analytic family of bounded linear operators in the right half plane.
For a function $u \in L^{\infty}(X)$ with $u(t)=0$ for all $t\leq 0$ on some Banach space $X$,
we define
\begin{align*}
  \left[K(\dd) u \right]^n:=\sum_{j=0}^{n}{K_j u\left(t_n - t_j\right)},
\end{align*}
where the weights $K_j$ are defined as the coefficients satisfying
$K\left(\frac{\delta(z)}{\Delta t} \right)=:\sum_{n=0}^{\infty}{K_n z^n}$.
\end{definition}
\begin{remark}
  We will use the same notation if $u:=(u_n)_{n\in\N}$ is a sequence of values in $X$, by identifying
  $u$ with the piecewise constant function. The connection with Definition~\ref{def:operational_calculus}
  can be seen by applying the $\mathscr{Z}$ transform to the discrete convolution:
  \begin{align*}
    \mathscr{Z}\left(K(\dd) u\right)&=K\left(\frac{\delta(z)}{\Delta t}\right) \mathscr{Z}(u).
  \end{align*} \eremk
\end{remark}
This operational calculus then implies the convolution quadrature discretization
of \eqref{eq:continuous_int_eq}, by replacing $B(\partial_t)$ with $B(\dd)$ resulting in the following problem.

\begin{problem}
\label{problem:fully_discrete_int_eq}
For all $n \in \N$, find $(\varphi,\psi):=(\varphi^n,\psi^n)_{n\in \N} \subseteq X_h \times Y_h$ such that:
\begin{align}
\label{eq:fully_discrete_int_eq}
  \dualproduct{\left[B_{\text{imp}}(\dd ) \colvec{\varphi \\ \psi}\right]^n}{\colvec{\xi \\ \eta}} + \dualproduct{g(\psi^n + \bdryinterpY\dot{u}^{inc}(t_n))}{\eta}
&= \dualproduct{-\partial^+_n u^{inc}(t_n)}{\eta}
 \quad  \quad \forall (\xi,\eta) \in X_h \times Y_h.
\end{align}
\hfill\qed
\end{problem}



Since we will often be working with pairs $(\varphi,\psi) \in H^{-1/2}(\Gamma) \times H^{1/2}(\Gamma)$ we define the product norm
\begin{align}
\label{eq:def_triplenorm}
  \triplenorm{(\varphi,\psi)}^2:=\norm{\varphi}_{H^{-1/2}(\Gamma)}^2 + \norm{\psi}_{H^{1/2}(\Gamma)}^2.
\end{align}

\section{Well posedness}
\label{sect:well_posedness}
In this section, we investigate the existence and uniqueness of solutions to Problem~\ref{problem:fully_discrete_int_eq}.
We start with some basic properties of  the operator induced by $g$ and the operator $B_{\text{imp}}$. 
\begin{lemma}
\label{lemma:g_is_bounded}
The operator $g: H^{1/2}(\Gamma) \to H^{-1/2}(\Gamma)$ is a 
\emph{bounded (nonlinear) operator},
with
\begin{align*}
  \norm{g(\eta)}_{H^{-1/2}(\Gamma)} &\leq C\left(1+ \norm{\eta}_{H^{1/2}(\Gamma)}^{p}\right),
\end{align*}
where $p$ is the bound from Assumption~\ref{assumption:nonlinearity_g}(\ref{assumption:nonlinearity_g_growth})
and the constant $C>0$ depends on $\Gamma$ and $g$.
\end{lemma}
\begin{proof}
  We note that the following Sobolev embeddings hold (see \cite[Theorem 7.57]{adams_sobolev_spaces}):
  \begin{align}
    \label{eq:sobolev_embedding}
    H^{1/2}(\Gamma) &\subseteq  L^{p'}(\Gamma) \quad 
    \begin{cases}
       \forall 1\leq p'<\infty & \text{ for } d=2, \\
       \forall 1\leq p'\leq \frac{2d-2}{d-2} & \text{ for } d\geq 3.
    \end{cases}
  \end{align}
  Let $p$ be as in Assumption~\ref{assumption:nonlinearity_g}~(\ref{assumption:nonlinearity_g_growth}).
  Fix $p',q'$ such that $1/p' + 1/q' = 1$ and 
  both $p'$ and $pq'$  are in the admissible range of the Sobolev embedding. The case $d=2$ is clear. For $d\geq 3$ we use $p'=\frac{2d-2}{d-2}$, $q':=\frac{2d-2}{d}$.
  
  For $\eta,\xi \in H^{1/2}(\Gamma)$ we calculate:
  \begin{align*}
    \int_{\Gamma}{g(\eta)\xi}
    &\leq \norm{g(\eta)}_{L^{q'}(\Gamma)} \norm{\xi}_{L^{p'}(\Gamma)}
    \lesssim \left(1+\norm{\eta}_{L^{q'p}(\Gamma)}^{p} \right) \norm{\xi}_{H^{1/2}(\Gamma)} \\
    &\lesssim \left(1+ \norm{\eta}_{H^{1/2}(\Gamma)}^{p} \right) \norm{\xi}_{H^{1/2}(\Gamma)}.
  \end{align*}  
\end{proof}

The operator $B_{\text{imp}}(s)$ is elliptic in the frequency domain:
\begin{lemma}
  \label{lemma:coercivity_calderon_ld}
  There exists a constant $\beta > 0$, depending only on $\Gamma$, such that
  \begin{align}
    \label{eq:corecivity_calderon_ld}
    \Re\dualproduct{B_{\text{imp}}(s) \colvec{\varphi \\ \psi }} {\colvec{\varphi \\ \psi}}
 &\geq \beta \min(1,\abs{s}^2)\frac{\Re(s)}{\abs{s}^2}  \triplenorm{(\varphi, \psi)}^2.
  \end{align}
\end{lemma}
\begin{proof}
  The analogous estimate to \eqref{eq:corecivity_calderon_ld} for the operator $B(s)$ was
  shown in \cite[Lemma 3.1]{bls_fembem}. Since the bilinear form induced by $B_{\text{imp}}(s)- B(s)$ is skew-hermitean, this implies
  \eqref{eq:corecivity_calderon_ld}. 
\end{proof}

The solvability of the discrete system (\ref{eq:fully_discrete_int_eq}) will be based on the theory of monotone operators. We summarize the 
main result in the following proposition
\begin{proposition}[Browder and Minty,{\cite[Chapter II, Theorem 2.2]{showalter_book}}]
\label{prop:browder_minty}
Let $X$ be a real separable and reflexive Banach space and $A: X \to X'$ be a bounded, continuous, coervice
and monotone map from $X$ to its dual space(not necessarily linear),
i.e.,  $A$ satisfies:
\begin{itemize}
  \item $A: X \to X'$ is continuous,
  \item the set $A(M)$ is bounded in $X'$ for all bounded sets $M \subseteq X$,
  \item $\displaystyle \lim_{\norm{u} \to \infty}{\frac{ \dualproduct[X' \times X]{ A(u)}{u}}{\norm{u}}} = \infty$,
  \item $\displaystyle \dualproduct[X'\times X]{A(u)-A(v)}{u-v} \geq 0$ for all $u,v \in X$.
\end{itemize}

Then the variational equation 
\begin{align*}
  \dualproduct[X' \times X]{A(u)}{v}&=\dualproduct[X' \times X]{f}{v},  \quad \quad \forall v \in X,
\end{align*}
has at least one solution for all $f \in X'$.
If the operator is strongly monotone, i.e., there exists $\beta > 0$ such that
\begin{align*}
  \dualproduct[X'\times X]{A(u)-A(v)}{u-v} &\geq \beta \norm{u-v}^2_{X} \quad \text{ for all $u,v \in X$},
\end{align*}
then the solution is unique.
\end{proposition}
\begin{proof}
  The first part is just a slight reformulation of \cite[Theorem 2.2]{showalter_book}, based on some of the
  equivalences stated in the same chapter. Uniqueness  follows by considering
  two solutions $u,v$ and applying the strong monotonicity to conclude $\norm{u-v}_X=0$.
\end{proof}

\begin{theorem}
  \label{thm:existence_discrete}
  Let $\Delta t >0 $ and $(X_h,Y_h) \subseteq H^{-1/2}(\Gamma) \times H^{1/2}(\Gamma)$ be closed subspaces. 
  Then the discrete system of equations \eqref{eq:fully_discrete_int_eq} has a unique solution
  in the space $X_h \times Y_h$ for all $n \in \N$.
\end{theorem}
\begin{proof}
  We prove this by induction on $n$. For $n=0$ we are given the initial condition $\varphi^0=\psi^0=0$. Assume we have solved
  \eqref{eq:fully_discrete_int_eq} up to the $n-1$-st step.
  We denote the operators from the definition of $B_{\text{imp}}(\dd)$ as $B_j$, $j\in \N_0$, 
  dropping the subscript.
  We set $\widetilde{\psi}^n:=\psi^n+\bdryinterpY \dot{u}^{inc}(t_n)$
  and bring all known terms to the right-hand side.
  Then, in the $n$-th step the equation reads
  \begin{align}
    \label{eq:int_eq_time_step}
    \dualproduct{\left[B_0
    \colvec{\varphi^n \\ \widetilde{\psi}^n}\right]^n}{\colvec{\xi \\ \eta}}
    + \dualproduct{g(\widetilde{\psi}^n) }{\eta} &= \dualproduct{f^n}{\colvec{\xi \\ \eta}},
  \end{align}
  with $\displaystyle f^n:=\colvec{0 \\ -\partial_n^+ u^{inc}(t_n)} - \sum_{j=0}^{n-1}{ B_{n-j} \colvec{\varphi^{j} \\\psi^j}}
  + B_0  \colvec{0 \\ \bdryinterpY\dot{u}^{inc}(t_n)}$.
  The right-hand side is a continuous linear functional 
  with respect to $(\xi,\eta)$ due to the mapping properties of the operators $B_j$ that are   easily transfered from  the frequency-domain versions~\eqref{eq:mapping_properties}; see \cite{lubich_94}.
  
  In order to apply Proposition~\ref{prop:browder_minty}, 
  we note that the operator $B_0: H^{-1/2}(\Gamma)\times H^{1/2}(\Gamma) \to H^{1/2}(\Gamma) \times H^{-1/2}(\Gamma)$
  is the leading term of a power series, and therefore $B_0=B_{\text{imp}}\left(\frac{\delta(0)}{\Delta t}\right)$.
  This implies $B_0$ is elliptic via Lemma \ref{lemma:coercivity_calderon_ld}. The nonlinearity satisfies: 
  $\dualproduct{g(\eta)}{\eta} = \int_{\Gamma}{g(\eta) \eta} \geq 0 $
  by Assumption~\ref{assumption:nonlinearity_g}~(\ref{assumption:nonlinearity_gss}). This implies that 
  the left-hand side in~\eqref{eq:int_eq_time_step} is coercive.
  For $\eta_1,\eta_2$,
  we apply the mean value theorem, to get:
  \begin{align*}
    \dualproduct{g(\eta_1)-g(\eta_2)}{\eta_1 - \eta_2} = \int_{\Gamma}{g'(s(x))(\eta_1(x) - \eta_2(x))^2 \;dx}
    &\geq 0,
  \end{align*}
  since $g' \geq 0$ via Assumption~\ref{assumption:nonlinearity_g}(\ref{assumption:nonlinearity_g_prime}).
  Thus the left-hand side in~\eqref{eq:int_eq_time_step} is also strongly monotone.
  We have already seen boundedness in Lemma~\ref{lemma:g_is_bounded}.
  The continuity is a consequence of Sobolev's embedding theorem, with the detailed proof  given later
  in more generality as part of 
  Lemma \ref{lemma:approx_nonlinearity}~(\ref{it:approx_nonlinearity_convergence_weak2}).
\end{proof}

\section{Convergence analysis} 
In this section we are interested in the convergence of the method towards the exact solution.
A straight forward approach would be to use the positivity of $B_{\text{imp}}$ and monotonicity of $g$ to bound the error in terms of a residual. 
Unfortunately, this approach necessitates strong assumptions on the regularity of the exact solution and seems only to give estimates in a rather weak norm;
see Appendix~\ref{appendix:herglotz} for a sketch of this methodology.
Instead of using the integral equation, we will show convergence by analysing an equivalent problem based on the approximation of the differential equation \eqref{eq:wave_eqn}.
This equivalence is spelled out in Lemma~\ref{lemma:diff_eq_u_ie}. We will then spend the rest of the section analysing the discretization errors between this formulation and the
exact solution.
The construction of the equivalent system is based on the idea of exotic transmission problems as introduced in \cite{laliena_sayas}.

For a space $X \subseteq Y$ let the annihilator $X^{\circ} \subseteq Y'$ be defined as
\begin{align*}
  X^{\circ}:=\left\{ f \in Y': f(x) = 0 \; \forall x \in X \right\}.
\end{align*}

\begin{lemma}
\label{lemma:diff_eq_u_ie}
  Let $\Delta t > 0$, and let $X_h \subseteq H^{-1/2}(\Gamma)$, $Y_h \subseteq H^{1/2}(\Gamma)$
  be closed subspaces.
  Let
  \begin{align}
    \label{eq:def_Hh}
    \HH&:=\{ u \in \Hpglobal{1}: \tracejump{u} \in Y_h, \gamma^- u \in X_h^\circ \}.
  \end{align}

  Consider the sequence of problems: Find $\uie^n, \vie^n \in \HH$  for $n  = 0,1,\dots$ such that
\begin{subequations}
  \label{eq:diff_eq_u_ie}
  \begin{align}
    \frac{1}{\Delta t}\sum_{j=0}^k{\alpha_j \uie^{n-j}}&=\vie^{n}  \\
    \frac{1}{\Delta t}\sum_{j=0}^k{\alpha_j \vie^{n-j}}&=\laplace \uie^{n}  \\    
    \partial^+_n \uie^n - g\left(\tracejump{\vie^{n}} + \bdryinterpY\dot{u}^{inc}(t_n)\right)  
                                                       &+\partial_n^+  u^{inc}(t_n) \in X_h^\circ, \\
    \normaljump{\uie} &\in X_h,
  \end{align}
  where $t_n:=n\Delta t$ and $\uie^{-j}=\vie^{-j}:=0$ for $j\in \N$.
\end{subequations}
  Then the following two statements hold:
  \begin{enumerate}[(i)]
    \item
      \label{it:diff_eq_u_ie_1}
      If the sequences $\varphi^n$, $\psi^n$ solve (\ref{eq:fully_discrete_int_eq}), 
      then $\uie:=S(\dd)\varphi + \left(\dd \right)^{-1}D(\dd ) \psi$ and 
      $\vie:= \dd \uie$ solve (\ref{eq:diff_eq_u_ie}).
    \item
      \label{it:diff_eq_u_ie_2}
      If $\uie$,$\vie$ solve (\ref{eq:diff_eq_u_ie}), then traces
      $\varphi:=-\normaljump{ \uie}$, $\psi:=\tracejump{\vie}$
      solve \eqref{eq:fully_discrete_int_eq}.
    \end{enumerate}
Note: the subindex ``$ie$'', which stands for ``integral equations'', is used to separate this sequence from the one
    obtained by applying the multistep method to the semigroup, as defined in \eqref{eq:def_usg_vsg}.
\end{lemma}
\begin{proof}
  We first note that (\ref{eq:diff_eq_u_ie}) has a solution in $\HH$. 
  
  We show this by induction on $n$. For $n\leq0$ we set $\uie^n:=\vie^n:=0$.
  For $n \in \N$, we consider the weak formulation, find $\uie^n \in \HH,$ $\vie^n \in \HH$, such that
  \begin{subequations}
    \label{eq:weak_form_proof_eq_u_ie}
    \begin{align}
      \frac{1}{\Delta t}\sum_{j=0}^k{\alpha_j \uie^{n-j}}
      &=\vie^{n}, \\
      \ltwoproduct{\frac{1}{\Delta t}\sum_{j=0}^k{\alpha_j \vie^{n-j}}}{z_h}
      &=\begin{aligned}[t]
        &-\ltwoproduct{\nabla \uie^{n} }{\nabla z_h}
- \dualproduct{
              g\left(\tracejump{\vie^{n}} + \bdryinterpY \dot{u}^{inc}(t_{n})\right)
              - \partial_n^+  u^{inc}(t_{n})}{\tracejump{z_h}},
        \end{aligned}
    \end{align}
  \end{subequations}
  for all $z_h \in \HH$.
  Multiplying the first equation by $\Delta t$ and collecting all the terms involving $\uie^j$ and $\vie^j$ for $j<n$ in  $F_n \in \HH$, 
  the condition becomes $\alpha_0 \uie^{n} = \Delta t \,\vie^{n} + F_n$. After inserting this identity and combining all 
  known terms into a new right-hand side $\widetilde{F}^n \in \HH'$, the second equation becomes
  \begin{align*}
    \frac{\alpha_0}{\Delta t}\ltwoproduct{ \vie^n }{z_h} + \frac{\Delta t}{\alpha_0}\ltwoproduct{ \nabla \vie^n}{\nabla z_h} + 
    \dualproduct{g\left(\tracejump{\vie^{n}} + \bdryinterpY \dot{u}^{inc}(t_{n})\right)}{\tracejump{z_h}} &= \dualproduct[\HH' \times \HH]{\widetilde{F}^n}{z_h}.
  \end{align*}

  Since $\HH$ is a closed subspace of $H^1$, this equation can be solved for all $n \in \N$ due to the 
  monotonicity of the operators involved and the Browder-Minty theorem; see Proposition \ref{prop:browder_minty} and also the proof of Theorem \ref{thm:existence_discrete} for how to treat the nonlinearity.
  With $\alpha_0 \uie^{n} := \Delta t \, \vie^{n} + F_n$, we have found a solution to \eqref{eq:weak_form_proof_eq_u_ie}.

  What still needs to be shown is that $\normaljump{\uie^n} \in X_h$.  
  Note that it is sufficient to show $\normaljump{\widehat{\uie}} \in X_h$ for the
  Z-transformed variable, as we can then express $\normaljump{\uie^n}$ as a Cauchy integral in $X_h$.
  The details of this argument are given later.

  It is easy to see that 
  $\norm{\uie^n}_{\HpLglobal{1}} \leq C(\Delta t) \sum_{j=0}^{n-1} {\norm{\uie^{j}}_{\HpLglobal{1}}}$, 
  where the constant may depend on $\Delta t$, but not on $\uie^{j}$, $\vie^{j}$ or $n$.
  This implies that the $Z$-transform $\widehat{\uie}(z)$ is well defined for $\abs{z}$ sufficiently
  small.

  To simplify notation, define
  $G_n:=g(\tracejump{\vie^{n}} + \bdryinterpY \dot{u}^{inc}(t_n)) - \partial_n  u^{inc}(t_n)$.  
  Taking the Z-transform of $(\uie)$ and $(\vie)$ a simple calculations shows that for $z_h \in \HH$
  \begin{align*}
    \ltwoproduct{\left(\frac{\delta(z)}{\Delta t}\right)^2 \widehat{\uie}}{z_h}
    + \ltwoproduct{\nabla \widehat{\uie}}{\nabla z_h} + \dualproduct{\widehat{G}}{\tracejump{z_h}} &=0.
  \end{align*}
  For $z_h \in C_0^{\infty}\left(\R^d \setminus \Gamma\right)$ this implies 
  \begin{align*}
    -\laplace \widehat{\uie} + \left(\frac{\delta(z)}{\Delta t}\right)^2 \widehat{\uie} &= 0.
  \end{align*}
  From $z_h \in \HH$ with $z_h|_{\Omega^-}=0$ we see $\partial_n^+ \widehat{\uie} -\widehat{G} \in Y_h^{\circ}$.
  Let $\xi \in X_h^\circ$ and $z_h$ is a lifting of $\xi$ to $\Hpglobal{1}$, 
  i.e., $\gamma^+z_h=\gamma^-z_h=\xi$, then
  we get by integration by parts:
  \begin{align*}
    \dualproduct{-\partial^+_n \widehat{\uie}}{\xi} + \dualproduct{\partial^-_n \widehat{\uie}}{\xi} &= 0,
  \end{align*}
  or $\normaljump{\widehat{\uie}} \in \left(X_h^{\circ}\right)^{\circ}=X_h$.
  We can use the Cauchy-integral formula to write:
  \begin{align*}
    \normaljump{\uie^n}&=\frac{1}{2\pi \ii} \int_{\mathcal{C}}{ \normaljump{\widehat{\uie}} z^{-n-1} \,dz},
  \end{align*}
  where the contour $\mathcal{C}:=\{z \in \C: \abs{z}= \text{const} \}$ denotes a sufficiently small circle,
  such that all the $Z$-transforms exist.  Since we have shown that $\normaljump{\widehat{\uie}} \in X_h$ and we assumed that $X_h$ is a closed space,
  this implies $\normaljump{u^n} \in X_h$. Thus we have shown the existence of a solution to (\ref{eq:diff_eq_u_ie}).

  We can now show the equivalence of (\ref{it:diff_eq_u_ie_1}) and (\ref{it:diff_eq_u_ie_2}).
  We start by showing that the traces of the solutions to (\ref{eq:diff_eq_u_ie}) solve the
  boundary integral equation.
  We have the following equation in the frequency domain:
  \begin{align*}
    -\laplace \widehat{\uie} + \left(\frac{\delta(z)}{\Delta t}\right)^2 \widehat{\uie} &= 0 \\
    \partial^+_n \widehat{\uie} - \widehat G \in Y_h^{\circ}.
  \end{align*}
  The representation formula tells us that we can write
  $\widehat{\uie}(z)=-S(z)\normaljump{\widehat{\uie}(z)} + D(z) \tracejump{\widehat{\uie}(z)}$.
  We set $s^{\Delta t}:=\frac{\delta(z)}{\Delta t} $,
  $\widetilde{\psi}(z):=s^{\Delta t} \tracejump{\widehat{\uie}}=\tracejump{\widehat{\vie}}$ 
  and $\widetilde{\varphi}(z):=-\normaljump{\widehat{\uie}}$.
 Multiplying the representation formula by  $s^{\Delta t}$ gives
  $s^{\Delta t} \widehat{\uie}(z)= s^{\Delta t} S(z) \widetilde{\varphi} + D(z) \widetilde{\psi}$.   
  Taking the interior trace $\gamma^-$ and testing with a discrete function $\xi_h \in X_h$ gives   
  \begin{align*}    
    0&=\dualproduct{s^{\Delta t}\gamma^- \widehat{\uie}}{\xi_h}
      =\dualproduct{s^{\Delta t} V(s^{\Delta t}) \widetilde{\varphi}}{\xi_h}
       +\dualproduct{\left( K(s^{\Delta t}) - 1/2 \right)\widetilde{\psi} }{\xi_h}.
  \end{align*}
  Analogously,  by starting from the original representation formula,
  taking the exterior normal derivative $\partial_n^+$,  and testing with $\eta_h \in Y_h$ we obtain that
  \begin{align*}
    \dualproduct{G }{\eta_h}
    &=\dualproduct{\partial_n^+ \widehat{\uie} }{\eta_h}
      =\dualproduct{\left( 1/2 - K^t(s^{\Delta t}) \right) \widetilde{\varphi}}{\eta_h}
       +\dualproduct{ \left(s^{\Delta t}\right)^{-1} W(s^{\Delta t}) \widetilde{\psi}}{\eta_h}.
  \end{align*}
  Together, this is just the Z-transform of (\ref{eq:fully_discrete_int_eq}).
  By taking the inverse Z-transform, we conclude that the traces $\tracejump{\uie}$ and $\normaljump{\vie}$
    solve (\ref{eq:fully_discrete_int_eq}). By the uniqueness of the solution via Theorem~\ref{thm:existence_discrete},
    this implies $\varphi^n=-\normaljump{\uie^n}$ and $\psi^n=\tracejump{\vie^n}$, which then shows (\ref{it:diff_eq_u_ie_2}). 

  For (\ref{it:diff_eq_u_ie_1}), we observe that due to the
  uniqueness of solutions to Helmholtz transmission problems
  $\uie$ defined via (\ref{eq:diff_eq_u_ie}) and $\uie$ defined via potentials 
  have the same Z-transform and therefore coincide also in the time domain.
\end{proof}

\subsection{The continuous problem and semidiscretization in space}
\label{sect:sd_in_space}
In this section we investigate the problem in a time-continuous setting.
We consider the case of discretization in space via a Galerkin method, inspired by the spaces appearing in Lemma~\ref{lemma:diff_eq_u_ie},
and also show existence and uniqueness of (\ref{eq:wave_eqn}) under the assumptions on $g$ made in Assumption \ref{assumption:nonlinearity_g}.
The continuous problem is treated as a special case of the space-semidiscrete problem, as it allows a tighter presentation of arguments, instead
of having to prove things twice. We also lay the foundation for the later treatement of the discretization in time by 
 introducing the right functional analytic setting in the language of nonlinear semigroups.

\subsubsection{Semigroups}
\label{subsect:semigroups}
We would like to use the large toolkit provided by the theory of (nonlinear) semigroups, 
a summary of which can, for example, be found in \cite{showalter_book}. 
In order to do so, we will rewrite the wave equation~\eqref{eq:wave_eqn} as a first order system. We introduce the new variable $v$ by setting $v:=\dot{u}$ to get
\begin{align*}
  \colvec{\dot{u}\\\dot{v}}&= \colvec{ v \\ \laplace u}, \\
  \partial_n^+ u &= g(\gamma^+ v + \dot{u}^{inc}) - \partial_n^+ u^{inc}.
\end{align*}

The following definition is at the centre of the
used theory.
\begin{definition}
  Let $H$ be a Hilbert space and $\AA: H \to H$ be a (not necessarily linear  or continuous)
  operator with domain $\domain(\AA)$. We call $\AA$ maximally monotone
  if it satisfies:
  \begin{enumerate}[(i)]
  \item $\displaystyle \left(\AA x -\AA y,x-y\right)_{H} 
    \leq 0$ \quad $\forall x,y \in \domain{\AA}$,
  \item $\operatorname{range}\left(I-\AA\right)=H$.
  \end{enumerate}
\end{definition}
\begin{remark}
  We follow the notation used in \cite{graber}. Other authors, e.g. \cite{nevanlinna}
  work with $-\AA$ instead. \eremk
\end{remark}

The following proposition summarizes the main existence result from the
theory of nonlinear semigroups (we focus on the case $u_0 \in \domain(\AA)$).
\begin{proposition}[K\={o}mura-Kato, {\cite[Proposition 3.1]{showalter_book}}]
\label{prop:komura_kato}
  Let $\AA$ be a maximally monotone operator on a Hilbert
  space $H$ with domain $\domain(\AA) \subseteq H$.
  
  For each $u_0 \in \domain(\AA)$ there exists a unique absolutely continuous function $u: [0,\infty) \to H$,
  such that:
  \begin{align*}
    \dot{u} &= \AA u \quad \text{and } \quad u(0)=u_0
  \end{align*}
  almost everywhere in $t$.
Further,  $u$ is Lipschitz with $\norm{\dot{u}}_{L^{\infty}\left((0,\infty);H\right)} \leq \norm{\AA u_0}_{H}$ and $u(t) \in \domain(\AA)$ for all $t \geq 0$.
\end{proposition}

Since we would like to use monotone operator techniques, we need an appropriate
functional analytic setting. We introduce the Beppo-Levi space \cite{DenL}
\begin{align}
\label{eq:def_BL1}
  BL^1&:=\left\{u \in H^1_{\text{loc}}(\R^d \setminus \Gamma): \ltwonorm{\nabla u}< \infty \right\}/
        \operatorname{ker}{\nabla} ,
\end{align}
equipped with the norm $\norm{u}_{BL^1}:=\norm{\nabla u}_{L^2(\R^d \setminus \Gamma)}$ and the corresponding inner product.
This space contains all functions that are $H^1$ on compact subsets up to, not necessarily the same, constants
in the exterior and interior domain with $L^2$-gradient in $\Omega^+$ and $\Omega^-$.

The functional analytic setting for our problem is laid out in the next theorem.
  We formulate the problem so  that it covers the continuous in time/discrete in space case. To obtain the continous problem, we just  set $X_h:=H^{-1/2}(\Gamma)$
  and $Y_h:=H^{1/2}(\Gamma)$. 
\begin{theorem}
  \label{thm:semigroup}
  Consider the space 
  $\mathcal{X}:=BL^1 \times L^2(\R^d)$ with the product norm and corresponding inner product and
the block operator
  \begin{align}
    \AA:=
    \begin{pmatrix}
      0 & I \\
      \laplace   & 0
    \end{pmatrix},
  \end{align}
  \begin{align}
    \label{eq:def_domain_A}
    \domain(\AA):=\bigg\{& (u,v) \in  BL^1 \times L^2(\R^d):
    \laplace u \in L^2(\R^d \setminus \Gamma), v \in \HH, 
      \normaljump{u} \in X_h, 
      \partial^+_n u - g(\tracejump{v})  \in Y_h^\circ
    \bigg\}.
  \end{align}
  Then $\AA$ is a maximally monotone operator on $\mathcal{X}$
  and generates a strongly continuous semigroup that solves 
  \begin{align}
    \label{eq:ode_semigroup}
    \colvec{\dot{u} \\ \dot{v}}  = \AA\colvec{u \\ v} , \quad
    u(0)=u_0,\;
    v(0)=v_0,
  \end{align}
for all initial data $\left(u_0,v_0\right) \in \domain(\AA)$.

If additionally  to $v_0 \in \HH$,  also $u_0 \in \HH$, then the solution satisfies
  \begin{enumerate}[(i)]
  \item
    \label{it:sg_solution_prop_1}
    $\left(u(t),v(t)\right)\in \domain(\AA)$,  as well as $u(t) \in \HH$ and $v(t) \in \HH$  for all $t > 0$.
  \item
    \label{it:sg_solution_prop_2}
    $u \in C^{1,1}\left([0,\infty), H^{1}(\R^d\setminus \Gamma) \right)$,
  \item
    \label{it:sg_solution_prop_3} 
    $\dot{u} \in L^{\infty}\left((0,\infty), \Hpglobal{1} \right)$,
  \item
   \label{it:sg_solution_prop_4}
    $\ddot{u} \in L^{\infty}\left((0,\infty), L^2(\R^d) \right)$.
  \end{enumerate}
Since a-priori $u$ is only fixed up to constants,  (\ref{it:sg_solution_prop_1})-(\ref{it:sg_solution_prop_4}) are meant in the sense that there exists a representation which satisfies these properties.   From now on, we
  will not preoccupy ourselves with this distinction and always use 
  this representant.
\end{theorem}
\begin{proof}
  We first show monotonicity. Let $x_1=(u_1,v_1)$, $x_2:=(u_2,v_2)$ in $\domain(\AA)$. Then
  \begin{align*}
    \innerproduct[BL^1 \times L^2]{\AA x_1 - \AA x_2}{x_1 - x_2}
    &=\innerproduct{\nabla v_1 - \nabla v_2}{\nabla u_1 - \nabla u_2} + 
      \innerproduct{\laplace u_1 - \laplace u_2}{v_1 - v_2} \\
    &= \dualproduct{\partial_n^- u_1 - \partial_n^- u_2}{\gamma^-v_1- \gamma^-v_2}
      - \dualproduct{\partial_n^+ u_1 - \partial^+_n u_2}{\gamma^+ v_1-\gamma^+v_2} \\
    &=-\dualproduct{\normaljump{u_1}-\normaljump{u_2}}{\gamma^- (v_1- v_2)}
      - \dualproduct{\partial_n^+ u_1 - \partial_n^+ u_2}{\tracejump{(v_1 - v_2)}} \\
    &=-\dualproduct{g(\tracejump{v_1}) - g(\tracejump{v_2})}{\tracejump{(v_1 - v_2)} } \\
    & \leq 0,
  \end{align*}  
   where in the last step, we used the definition of the domain of $\AA$,
   which contains the
   boundary conditions and the fact that $\normaljump{u_j} \in X_h$.
  The definition of $\HH$ from \eqref{eq:def_Hh} gives that
  $\tracejump{(v_1 - v_2)} \in Y_h$.

  Next we show $\operatorname{range}(I-\AA)=\mathcal{X}$, i.e., for $(x,y) \in \mathcal{X}$ 
  we have to find $U=(u,v) \in \domain(\AA)$ such that $U-\AA U = (x,y)$.
  In order to do so, we first assume $x \in \HpLglobal{1}$
  (a dense subspace of $BL^1$).  
  From the first equation, we get $u-v=x$, or $u=v+x$, which makes the second equation:
  $v - \laplace v = y + \laplace x$. For the boundary conditions this gives us the requirements
  \begin{align*}
    \gamma^- v &\in X_h^{\circ}, &  \tracejump{v} \in Y_h, \\
    \normaljump{u} &\in X_h, &\partial_n^+ v - g(\tracejump{v}) + \partial_n^+ x \in Y_h^{\circ} .
  \end{align*}
  This can be solved analogously to the proof of Lemma~\ref{lemma:diff_eq_u_ie}.
  The weak formulation is: find $v \in \HH$, such that for all $w_h \in \HH$
  \begin{align*}
    \ltwoproduct{v}{w_h} + \ltwoproduct{\nabla v}{\nabla w_h} + \dualproduct{g(\tracejump{v})}{\tracejump{w_h}}
    &=\ltwoproduct{y}{w_h} - \ltwoproduct{\nabla x}{\nabla w_h} \quad \forall w_h \in
                                                            \HH.
  \end{align*}
  Due to the monotonicity of the left-hand side this problem has a solution via
  Proposition~\ref{prop:browder_minty}. We then set $u = v - x $. 
  The fact that the the conditions on $\normaljump{u}$ hold follow from
  the same argument as
  in Lemma~\ref{lemma:diff_eq_u_ie}, using $(X_h^\circ)^{\circ} = X_h$. 
  We therefore have $(u,v) \in \domain(\AA)$.
  For general $X:=(x,y) \in \mathcal{X}$ we argue via a density argument. Let $X_n:=(x_n,y_n)$ be a sequence
  in $\mathcal{X} \cap \left(\HpLglobal{1} \times L^2(\R^d) \right)$ such that $X_n \to X$.
  Let $U_n:=(u_n,v_n)$ be the respective solutions to $(I - \AA)U_n=X_n$.
  From the monotonicity of $\AA$, we easily see that for $n,m \in \N$:
  $\norm{U_n - U_m}_{\mathcal{X}}\leq \norm{X_n - X_m}_{\mathcal{X}}$, which means $(U_n)$ is 
  Cauchy and converges to some $U=:\left({u},{v}\right)$.
  From the first equation $u_n - v_n = x_n$ we get that $v_n \to v$ in $\Hpglobal{1}$.
  From the second equation $v_n -\laplace u_n=y_n$we get $\laplace u_n \to \laplace u$ in $L^2$.
  Therefore we have $u_n \to u$ in $\HpLglobal{1}$, which implies $\normaljump{u_n} \to \normaljump{u} \in X_h$.
  From Lemma \ref{lemma:approx_nonlinearity}~(\ref{it:approx_nonlinearity_convergence_weak})
  we get $\dualproduct{g(u_n)}{\xi} \to \dualproduct{g(u)}{\xi}$, 
  which implies $\partial^+_n u - g(u) \in Y_h^\circ$. 
  The other trace conditions follow from the $H^1$-convergence of $v_n$.
  The existence of the semigroup then follows from the K\={o}mura-Kato theorem;
  see Proposition \ref{prop:semigroup_approx}.

  In order to see that the additional assumption on the intial data implies $u(t) \in \HH$, we look at the differential equation
  and integrate to obtain
  \begin{align*}
    \dot{u} &= v \quad \Rightarrow u(t)= \int_0^t{v(s) \;ds} + u(0).
  \end{align*}
  Since $\HH$ is a closed subspace and $u_0 = u(0) \in \HH$ it follows that $u(t) \in \HH$.

  The regularity results can be directly seen from the differential equation.
  We remark that the statement $u(t) \in H^1\left(\R^d \setminus \Gamma\right)$ instead of $BL^1$ is meant
  in the sense of ``we can choose a representative in the equivalence class''.
\end{proof}

\begin{remark}
With the choice $X_h:=H^{-1/2}(\Gamma)$ and $Y_h:=H^{1/2}(\Gamma)$, the semigroup of Lemma
  \ref{thm:semigroup} represents the exact solution for \eqref{eq:wave_eqn}. \eremk
\end{remark}

\subsection{Approximation theory in $\HH$}
\label{subsect:approximation_HH}
\label{sect:approx_in_hh}
In this section we investigate the properties of the spaces $\HH$ and introduce some
projection/quasi-interpolation operators required in the analysis.

We start by defining an operator, which in some sense represents a ``volume version'' of
$\bdryinterpY$; see Lemma~\ref{lemma:prop_of_pi_0}~(\ref{it:pi_0_trace_relation}).
\begin{definition}
\label{def:trace_ltwo_projector}
Let $\mathcal{L}: H^{1/2}(\Gamma) \to \Hpglobal{1}$ denote the continuous lifting operator, 
such that $\gamma^+\mathcal{L}v=v$ and $\mathcal{L} v = 0$ in $\Omega^-$.
Then we define the operator $\Pi_0$ as
\begin{align*}
  \Pi_0: \left\{ u \in \Hpglobal{1}\,:\, \gamma^-u = 0\right\} &\to \left(\HH,\norm{\cdot}_{\Hpglobal{1}}\right) \\
  v &\mapsto        v - \mathcal{L}\left( \left(I-\bdryinterpY \right)\gamma^+v\right) & \text{in } \R^d\setminus\Gamma.
\end{align*}
Recall in the above that $\bdryinterpY: H^{1/2}(\Gamma) \to (Y_h,\norm{\cdot}_{1/2})$ denotes a stable operator.
\end{definition}

In the next lemma, we collect some of the most important properties of $\Pi_0$.
\begin{lemma}
\label{lemma:prop_of_pi_0}
  The following statements hold:
  \begin{enumerate}[(i)]
   \item if $\bdryinterpY$ is a projection, then $\Pi_0$ is a projection,
   \item 
     \label{it:pi_0_trace_relation}
     $\displaystyle \tracejump{\Pi_0 u}= \bdryinterpY \tracejump{u}$ ,
   \item $\Pi_0$ is stable, with the constant depending only on $\Gamma$ and
     $\norm{\bdryinterpY}_{H^{1/2}(\Gamma) \to H^{1/2}(\Gamma) }$,
   \item $\Pi_0$ has the same approximation properties in the exterior domain as $\bdryinterpY$ on $\Gamma$, i.e.
     \begin{align*}
       \norm{u - \Pi_0 u}_{H^{1}(\Omega^+)}&\leq C \norm{\tracejump{u}-\bdryinterpY \tracejump{u}}_{H^{1/2}(\Gamma)}.
     \end{align*}
  \end{enumerate}
\end{lemma}
\begin{proof}
  All the statements are immediate consequences of the definition and the continuity
  of $\mathcal{L}$ and $\bdryinterpY$.
\end{proof}

In the analysis of time-stepping methods, the Ritz-projector plays a major role.
For our functional-analytic setting it takes the following form.
\begin{definition}
\label{def:ritz_projector}
Let $\alpha > 0$ be a fixed stabilization parameter. Define the Ritz-projector $\Pi_1$ as
\begin{align*}
  \Pi_1: \HpLglobal{1} &\to \HH, 
\end{align*}
where $\Pi_1 u \in \HH$ is the unique solution to
\begin{align*}
  \ltwoproduct{\nabla \Pi_1 u}{\nabla v_h} + \alpha \ltwoproduct{\Pi_1 u}{v_h}
  &= \ltwoproduct{\nabla u}{\nabla v_h} + \alpha \ltwoproduct{ u}{v_h} 
    + \dualproduct{\normaljump{u}}{\gamma^-v_h} \quad  \forall v_h \in \HH.
\end{align*}
\end{definition}

\begin{lemma}  
  The operator $\Pi_1$ has the following properties:
  \begin{enumerate}[(i)]
    \item $\Pi_1$ is a stable projection onto the space
      $\HH \cap \left\{u \in \HpLglobal{1}: \normaljump{u} \in X_h\right\}$
      with respect to the $\Hpglobal{1}$-norm.
    \item $\Pi_1$ almost reproduces the exterior normal trace:
      \begin{align*}
        \dualproduct{\partial_n^+ \Pi_1 u}{\xi} &=\dualproduct{\partial_n^+ u}{\xi}, \quad \forall \xi \in Y_h.
      \end{align*}
    \item $\Pi_1$ has the following approximation property:
  \begin{align*}
    \norm{ (I-\Pi_1) u}_{\Hpglobal{1}}
    &\leq C \left(\inf_{u_h \in \HH} {\norm{u-u_h}_{\Hpglobal{1}} }
      + \inf_{x_h \in X_h}{\norm{\normaljump{u} - x_h}_{H^{-1/2}(\Gamma)} } \right).
  \end{align*}
  \end{enumerate}
  For $u \in H^1(\R^d \setminus \Gamma)$, with $u|_{\Omega^-}=0$,
  this approximation problem can be reduced to the boundary spaces $X_h$, $Y_h$:
  \begin{align*}
    \norm{ (I-\Pi_1) u}_{\Hpglobal{1}}    
    &\leq C_1 \inf_{y_h \in Y_h}{\norm{\gamma^+u - y_h}}_{H^{1/2}(\Gamma)}
      + C_2 \inf_{x_h \in X_h}{\norm{\normaljump{u} - x_h}_{H^{-1/2}(\Gamma)} }.
  \end{align*}
  All the constants depend only on $\Gamma$ and $\alpha$.
\end{lemma}
\begin{proof}
  This operator is well defined and stable
  as $\HH$ is a closed subspace of $H^{1}(\R^d \setminus \Gamma)$ and the bilinear form used is elliptic. 
  The fact that $\Pi_1$ reproduces the normal jump follows from testing with $v \in C_0^{\infty}(\R^d)$
  to get the partial differential equation and then using an arbitrary $v$  together with integration by parts.
  In order to see $\normaljump{\Pi_1 u} \in X_h$ we follow the same argument
  as in the proof of Lemma~\ref{lemma:diff_eq_u_ie}. For $\xi \in X_h^\circ$
  we obtain by using a global $H^1$-lifting and integration by parts:
  \begin{align*}
    \dualproduct{\normaljump{\Pi_1 u}}{\xi} &=\dualproduct{\normaljump{u}}{\xi}
                                              - \dualproduct{\normaljump{u}}{\xi} =0.
  \end{align*}
  Thus $\normaljump{\Pi_1 u} \in \left(X_h^{\circ}\right)^{\circ} =X_h$ and
  $\operatorname{range}(\Pi_1) \subseteq \HH \cap \left\{u \in \HpLglobal{1}: \normaljump{u} \in X_h\right\}$.
  The fact that it is a projection follows from the fact that for $u \in \HH$ with
  $\normaljump{u} \in X_h$, the term $\dualproduct{\normaljump{u}}{\gamma^- v_h}$ vanishes.

  For any $u_h \in \HH$ and $x_h \in X_h$
  \begin{align*}
    \norm{ (I-\Pi_1) u}^2_{\Hpglobal{1}}
    &\lesssim \ltwoproduct{\nabla (I- \Pi_1) u}{\nabla (I-\Pi_1) u} 
      + \alpha \ltwoproduct{ (I-\Pi_1) u}{(I-\Pi_1) u}  \\
    &=\ltwoproduct{\nabla (I- \Pi_1) u}{\nabla u - \nabla u_h}
      + \alpha \ltwoproduct{ (I-\Pi_1) u}{ u - u_h}   
    + \dualproduct{\normaljump{u} - x_h}{\gamma^- \Pi_1 u - \gamma^- u_h} \\
    & \lesssim\begin{multlined}[t]
      \Hpnorm{1}{ (I - \Pi_1)u} \Hpnorm{1}{u-u_h} \\
      +\norm{\normaljump{ u} - x_h}_{H^{-1/2}(\Gamma)}\left(\Hpnorm{1}{u - \Pi_1 u} +
      \Hpnorm{1}{u - u_h} \right).\end{multlined}
  \end{align*}
  Young's inequality concludes the proof.
  
  For the second part, we need to estimate $\inf_{u_h \in \HH} {\norm{u-u_h}_{\Hpglobal{1}}}$.
  Let $y_h \in Y_h$ be arbitrary and let $\theta$ be a continuous $H^1$ lifting of $ \gamma^+ u-y_h$ to
  $\Omega^+$.
  Define $u_h \equiv 0$ in $\Omega^-$ and $u_h:=u-\theta$ in $\Omega^+$. Then we have by construction 
  $\tracejump{u_h}=y_h \in Y_h$ and therefore $u_h \in \HH$. For the norm we estimate:
  \begin{align*}
    \norm{ u -u_h}_{\Hpglobal{1}} &= \norm{ \theta}_{H^1(\Omega^+)} \leq C \norm{ \gamma^+ u - y_h}_{H^{1/2}(\Gamma)}.
  \end{align*}
  
\end{proof}

Since we are interested in the case of low regularity, it is often not enough to 
have $H^1$ stable projection operators. For this special case we need to make
an additional assumption on $Y_h$.
\begin{assumption}
\label{ass:X_tilde}
  For all $h >0$, there exist spaces ${Y}_h^{\Omega^-} \subseteq H^1(\Omega^-)$ such that  $Y_h \supseteq \gamma^- Y_h^{\Omega^-}$ and that there exists  a linear operator
  $\volumeinterpY: L^2(\Omega^-) \to {Y}_h^{\Omega^-}$ with the following properties: $\volumeinterpY$ is stable in the $L^2$ and $H^1$ norm
  and for $s=0,1$ satisfies the strong convergence
  \begin{align*}
    \norm{u - \volumeinterpY u}_{H^s(\Omega^-)} & \to  0  \quad \text{for $h\to 0$, } \; \forall u \in H^{s}(\Omega^-).
  \end{align*}
\end{assumption}
This allows us to define our final operator.
\begin{lemma}
\label{lemma:ltwo_stable_projection}
  Let $\mathcal{E}^{\pm}: H^{m}\left(\Omega^{\pm}\right) \to H^m(\R^d)$ denote the Stein extension operator
  from \cite[Chap.~{VI.3}]{stein70}, which is stable for all $m \in \N_0$.
  Then we define a new operator
  \begin{align*}
    \Pi_2 : L^2(\Omega^{+}) &\to \HH \\
    u&\mapsto u - \mathcal{E}^-\left(\left(I - \volumeinterpY \right) \mathcal{E}^+ u\right)
    \quad \text{in $\Omega^+$}
  \end{align*}
  and $\Pi_2 u:=0$ in $\Omega^-$; i.e.\ in order to get a correction term similar
  to the one for $\Pi_0$ we extend the function to the interior,project/interpolate
  to ${Y}_h^{\Omega^-}$ and extend it back outwards.

  This operator has the following nice properties:
  \begin{enumerate}[(i)]
    \item $\Pi_2$ is stable in $L^2$ and $H^1$,
    \item for $s \in [0,1]$:
      \label{it:approximation_of_pi2}
      $\displaystyle \Hpnorm{s}{u - \Pi_2 u }\lesssim \Hpnorm{s}{ (I-\volumeinterpY) \mathcal{E}^+ u }$,
    \item $\Pi_2 u \to u$ in $H^1$ and $L^2$ for $h \to 0$, without further regularity assumptions.      
  \end{enumerate}
\end{lemma}
\begin{proof}
  In order to see that $\Pi_2 u \in \HH$, we have to show $\gamma^+ \Pi_2 u \in Y_h$. We calculate: 
  \begin{align*}
    \gamma^+ \Pi_2 u&=\gamma^+ u - \gamma^+ u  + \gamma^- \volumeinterpY \mathcal{E}^+u \in Y_h.
  \end{align*}
  For the approximation properties, we use the continuity of the Stein extension:
  \begin{align*}
    \Hpnorm{s}{u - \Pi_2 u}
    &=\Hpnorm{s}{\mathcal{E}^- \left(\left(I - \volumeinterpY\right) \mathcal{E}^+ u\right)}
    \lesssim \Hpnorm{s}{\left(I - \volumeinterpY \right) \mathcal{E}^+ u}.
  \end{align*}
  The extension $\mathcal{E}^+ u$ has the same regularity as $u$, thus we 
  end up with the correct convergence rates of ${Y}_h^{\Omega^-}$.
\end{proof}

\begin{remark}
  While Assumption~\ref{ass:X_tilde} looks somewhat artificial, in most cases it is easily verified, 
  as we can construct a ``virtual'' triangulation of $\Omega^-$ with piecewise polynomials.
  The projection operator $\volumeinterpY$ is then given by the (volume averaging based) Scott-Zhang  operator or for high order methods some quasi-interpolation operator (e.g. \cite{melenk_hp_interpolation}). \eremk
\end{remark}
\begin{remark}
  The use of a space on $\Omega^-$ is arbitrary and made to reflect the fact that usually  a natural triangulation on $\Omega^-$  is given. An artificial layer of triangles around $\Gamma$ in $\Omega^+$ could have been used instead and would have allowed us to drop the extension operator to the interior. \eremk
\end{remark}

To conclude this section, we investigate the convergence property of the nonlinearity $g$.
\begin{lemma}
\label{lemma:approx_nonlinearity}
Let $\eta \in H^{1/2}(\Gamma)$ and $\eta_h \in H^{1/2}(\Gamma)$ be such that
  $\eta_h$ converges to $\eta$  weakly, i.e., $\eta_h \rightharpoonup \eta$ for $h \to 0$. 
  Then the following statements hold:
  \begin{enumerate}[(i)]
    \item
      \label{it:approx_nonlinearity_convergence_weak}
      $g(\eta_h) \rightharpoonup g(\eta)$ in $H^{-1/2}(\Gamma)$.
    \item
      \label{it:approx_nonlinearity_convergence_weak2}
      If $\eta_h \to \eta$ strongly in $H^{1/2}(\Gamma)$, then
      $g(\eta_h) \to g(\eta)$ in $H^{-1/2}(\Gamma)$, i.e., 
      the operator $g: H^{1/2}(\Gamma) \to H^{-1/2}(\Gamma)$ is continuous.
    \item 
      \label{it:approx_nonlinearity_convergence_strong}
      Assume that $\eta_h \to \eta$ strongly in $H^{1/2}(\Gamma)$ and $|g(\mu)|\leq C (1+|\mu|^{p})$ with $p < \infty$ for $d=2$ and
      $p \leq \frac{d-1}{d-2}$, then the convergence is strong in $L^2$:
      \begin{align*}
        \|{g(\eta) -g(\eta_h)\|}_{L^2(\Gamma)} \to 0.
      \end{align*}
    \item
      \label{it:approx_nonlinearity_est_with_growth}
      Assume $|g'(s)|\leq C_{g'} (1+|s|^{q})$, 
      where $q <\infty$ is arbitrary for $d=2$, and $q\leq 1$ for $d=3$.
      Then we have the following estimates:
      \begin{align*}
        \|g(\eta) - g(\eta_h)\|_{L^2(\Gamma)}
        &\leq C(\eta) \|\eta - \eta_h\|_{L^{2+\varepsilon}(\Gamma)} & \text{ for $d=2,\; \forall \varepsilon > 0$},\\
        \|g(\eta) - g(\eta_h)\|_{L^2(\Gamma)} 
        &\leq C(\eta) \|\eta - \eta_h\|_{H^{1/2}(\Gamma)},
      \end{align*} 
      where the constant $C(\eta)$ does not depend on $h$.
    \item
      \label{it:approx_nonlinearity_est_with_l_infty}
      Alternatively, if $\eta \in L^{\infty}(\Gamma)$ and $\|\eta_h - \eta\|_{L^{\infty}} \leq C_{\infty}$ 
      we have
      \begin{align*}
        \|g(\eta) - g(\eta_h)\|_{L^2(\Gamma)}
        &\leq C(\|\eta\|_{\infty},C_{\infty}) \|\eta-\eta_h\|_{L^{2}(\Gamma)}.
      \end{align*}
  \end{enumerate}
\end{lemma}
\begin{proof}
  We focus on the case $d\geq 3$, the case $d=2$ follows along the same lines but is simpler since the Sobolev embeddings
  hold for arbitrary $p \in [1,\infty)$.

  Ad (\ref{it:approx_nonlinearity_convergence_weak}):  
  Since weakly convergent sequences are bounded (see~\cite[Theorem  1(ii), Chapter V.1]{yosida_fana}),
  we can apply Lemma~4.1 to get that 
  $g(\eta_h)$ is uniformly bounded in $H^{-1/2}(\Gamma)$. By the Eberlein-\v{S}mulian theorem, see~\cite[page~141]{yosida_fana},
  this implies that there exists a subsequence $g(\eta_{h_j})$, $j\in \mathbb{N}$, which converges weakly to some limit $\widetilde{g} \in H^{-1/2}(\Gamma)$.
  We need to identify the limit $\widetilde{g}$ as $g(\eta)$. By Rellich's theorem, the sequence $\eta_h$ converges to $\eta$ in $H^{s}(\Gamma)$ for $s < 1/2$,
  and using Sobolev embeddings we get that $\eta_{h_j} \to \eta$ in $L^{p'}(\Gamma)$ for $p' < \frac{2d-2}{d-2}$.
  Standard results from measure theory (e.g.~\cite[ Theorem IV.9]{brezis})
  then give (up to picking another subsequence) that $\eta_{h_j} \to \eta$ pointwise almost everywhere and there 
  exists an upper bound $\zeta \in L^{p'}(\Gamma)$ such that $|\eta_{h_j}|\leq |\zeta|$ almost everywhere.
  The growth conditions on $g$ are such that $C(1+|\zeta|^p)$ is an integrable upper bound of $g(\eta_h)$.
  By the continuity of $g$ we also get $g(\eta_{h_j}) \to g(\eta)$ almost everywhere.
  For test functions $\phi \in C^{\infty}(\Gamma)$, we get:
  \begin{align*}
    \int_{\Gamma}{g(\eta_{h_j}) \phi} \to \int_{\Gamma}{g(\eta) \phi}
  \end{align*}
  by the dominated convergence theorem (note that $\phi$ is bounded). On the other hand, since $C^{\infty}(\Gamma) \subseteq H^{1/2}(\Gamma)$,
  we get $\left<g(\eta_{h_j}),\phi\right>_{\Gamma} \to \left<\widetilde{g},\phi\right>_{\Gamma}$ due to the weak convergence.
  Since $C^{\infty}(\Gamma)$ is dense in $H^{1/2}(\Gamma)$, we get $g(\eta)=\widetilde{g}$. 
  This proof can be repeated for every subsequence, thus the whole sequence must converge weakly to $g(\eta)$.
  
  Ad~(\ref{it:approx_nonlinearity_convergence_weak2}):
  By the Sobolev embedding theorem, we have
  \begin{align*}
    \|g(\eta) - g(\eta_h)\|_{H^{-1/2}(\Gamma)}
    &=\sup_{0\neq v \in H^{1/2}(\Gamma)}{\frac{\left<{g(\eta) - g(\eta_h),v}\right>_{\Gamma}}{\|{v}\|_{H^{1/2}(\Gamma)}}} \\
    &\lesssim \|{g(\eta) - g(\eta_h)}\|_{L^{p'}(\Gamma)}.
  \end{align*}
  for $p':=(2d-2)/d$ for $d\geq 3$, and $p'>1$ for $d=2$.
  
  If $\eta_h \to \eta$ in $H^{1/2}(\Gamma)$, the Sobolev-embeddings give $\eta_h \to \eta$ in $L^{q}(\Gamma)$ for $q\leq \frac{2d-2}{d-2}$.  
  Again by [Bre83, Theorem IV.9], this implies that there exists a sub-sequence $\eta_{h_j}$, which converges pointwise almost everywhere
  and there exists a function $\zeta \in L^{q}(\Gamma)$ such that $|\eta_h|\leq |\zeta|$ almost everywhere.
  From the growth conditions on $g$, we get that $(1 + |\zeta|^{p}) \in L^{p'}(\Gamma)$.
  Since $g$ is continuous by assumption,
  $g(\eta_{h_j})$ converges to $g(\eta)$ pointwise almost everywhere. By the dominated convergence theorem
  this implies  $\int_{\Gamma}{|g(\eta_{h_j}) - g(\eta)|^{q'}} \to 0$.
  The same argument can be applied to show that every sub-sequence of $g(\eta_h)$
  has a sub-sequence that converges to $g(\eta)$ in $H^{-1/2}(\Gamma)$. This is sufficient to show that the whole sequence converges.

  Ad~(\ref{it:approx_nonlinearity_convergence_strong}):
  Follows along the same lines as~(\ref{it:approx_nonlinearity_convergence_weak2}). Instead of estimating the $H^{-1/2}$-norm by the $p'$ norm via the duality argument,
  we can directly work in $L^2$.
  Due to our restrictions on $g$ and the Sobolev embedding, we get an
  upper bound $C (1 + |\zeta|^p)$ in $L^2(\Gamma)$, which allows us to apply the same argument as before to get convergence.

  Ad~(~\ref{it:approx_nonlinearity_est_with_growth}):
  Using the growth condition on $g'$ we estimate for fixed $x \in \Gamma$:
  \begin{align}    
    |g(\eta(x)) - g(\eta_h(x))|
    &= \left|\int_{\eta_h(x)}^{\eta(x)}{ g'(\xi) \,d\xi}\right|\leq \left|\eta_h(x) - \eta(x)\right|\sup_{\xi \in [\eta_h(x), \eta(x)]} {|g'(\xi)|} \nonumber\\ 
    &\leq \left|\eta_h(x) - \eta(x)\right| C\left(1+\max\left(|\eta_h(x)|^q,|\eta(x)|^q\right)\right).      \label{nlw:eq:approx_nonlinearity_int1}
  \end{align}
  In the case $d=3$, we use the Cauchy-Schwarz inequality to estimate:
  \begin{align*}
    \|g(\eta) - g(\eta_h)\|_{L^2(\Gamma)}^{2}
    &\lesssim \|\left(1+\max(|\eta_h|^q,|\eta|^q)\right)^2\|_{L^2(\Gamma)} \|\left(\eta - \eta_h\right)^{2}\|_{L^2(\Gamma)} \\
    &\lesssim \left(1+ \|\eta_h\|_{L^{4q}(\Gamma)}^2 + \|{\eta}\|_{L^{4q}(\Gamma)}^2\right)\|{\left(\eta - \eta_h\right)}\|^2_{L^4(\Gamma)} \\
    &\lesssim \left(1+ \|\eta_h\|_{H^{1/2}(\Gamma)}^2 + \|{\eta}\|_{H^{1/2}(\Gamma)}^2\right)\|{\left(\eta - \eta_h\right)}\|^2_{H^{1/2}(\Gamma)},
  \end{align*}
  where in the last step we used the Sobolev embedding.
  Since weakly  convergent sequences are bounded, the first term can be uniformly bounded with respect to $h$, which shows~(iv) for $d\geq 3$.

  In the case $d=2$, we have by Sobolev's embedding that $\|{\max(|{\eta_h}|,|\eta|)}\|_{L^{p'}(\Gamma)}$ can be 
  bounded independently of $h$ for arbitrary  $p' > 1$. Using~\eqref{nlw:eq:approx_nonlinearity_int1} to 
  estimate the difference and applying Hölders inequality then proves (iv) in the case $d=2$.  

  Ad (\ref{it:approx_nonlinearity_est_with_l_infty}):
  We just remark that, since $g$ is assumed $C^1$, we have that $g'$ is bounded on compact  subsets of $\mathbb{R}$. Arguing as before, 
  the supremum   $\sup_{\xi}({g'(\xi))}$
  is therefore uniformly bounded, from which the stated result follows by applying Hölder's inequality and the Sobolev embedding.
\end{proof}

\subsubsection{Convergence of  the semidiscretization in space}
\label{subsect:convergence_sd_in_space}
We now focus on the discretization error, due to the spaces $X_h$ and $Y_h$.
In order to do so, we recast the semigroup solutions from Lemma \ref{thm:semigroup} into
a weak formulation.
\begin{lemma}
  Let $\HH$ be defined as in Lemma \ref{lemma:diff_eq_u_ie}. Then the
  semigroup solution \eqref{eq:ode_semigroup} denoted by $(u_h,v_h)(t)$ solves
  
  \begin{subequations}
    \label{eq:weak_form_sg_space_d}
  \begin{align}
    \ltwoproduct{\nabla \dot{u}_h(t)}{\nabla w_h}&=\ltwoproduct{\nabla v_h(t)}{\nabla w_h}, \\
    \ltwoproduct{\dot{v}_h(t)}{z_h} &=-\ltwoproduct{\nabla u_h(t)}{\nabla z_h} 
                                   -\dualproduct{g(\tracejump{v_h})}{\tracejump{z_h}},
  \end{align}
  \end{subequations}
  for all $(w_h,z_h) \in BL^1 \times \HH$ and $t> 0$.
  The exact solution satisfies
  \begin{subequations}
    \label{eq:weak_form_sg_space_c}
  \begin{align}
    \ltwoproduct{\nabla \dot{u}(t)}{\nabla w_h}&=\ltwoproduct{\nabla v(t)}{\nabla w_h}, \\
    \ltwoproduct{\dot{v}(t)}{z_h} &=-\ltwoproduct{\nabla u(t)}{\nabla z_h} 
                                 -\dualproduct{g(\tracejump{v})}{\tracejump{z_h}} 
                                 - \dualproduct{\normaljump{u}}{\gamma^- z_h},
                                 \label{eq:weak_form_sg_space_c_eq2}
  \end{align}
  \end{subequations}
  for all $(w_h,z_h) \in BL^1 \times \HH$ and $t > 0$.
\end{lemma}
\begin{proof}
  This is a simple consequence of \eqref{eq:diff_eq_u_ie}, the definition of $\domain(\AA)$ and integration by parts.
  We just note that the last term in \eqref{eq:weak_form_sg_space_c_eq2}, not there in a straight-forward weak formulation of the exterior problem \eqref{eq:wave_eqn}, appears
  because we replaced $\gamma^+ z_h$ with $\tracejump{z_h}$ in the boundary term containing the nonlinearity.
  The difference to \eqref{eq:weak_form_sg_space_d} with
  $X_h$ and $Y_h$ as the full space, is because of the condition $\gamma^- z_h \in X_h^{\circ}$, which
  would imply $\gamma^-z_h = 0$ for the full space case.
\end{proof}

Now that we have developed the appropriate approximation theory in the space $\HH$, we are able to quantify the convergence of the 
semidiscrete solution to the continuous one. This is the content of the next theorem.

\begin{theorem}
\label{thm:full_convergence_low_regularity}
  Assume that there exists an $L^2$-stable projection $\Pi_2$ onto $\HH$
  with $\ltwonorm{u - \Pi_2 u} \to 0$ for $h \to 0$
  as described in Lemma~\ref{lemma:ltwo_stable_projection}.
  
  Introducing the error functions 
  \begin{align*}
    \rho(t)&:=\colvec{\rho_u(t) \\ \rho_v(t)}:=\colvec{u - \Pi_1 u \\v - \Pi_2 v}, \\
    \theta(t)&:=g(\tracejump{v}(t))-g(\tracejump{\Pi_2 v(t)}),
  \end{align*}
  the convergence rate can be quantified as
  \begin{align}
    \label{eq:err_est_expl_projections}
    \ltwonorm{v_h(t) - v(t)}^2 &+ \ltwonorm{\nabla{u_h(t}) - \nabla{u}(t)}^2 +
    \beta \int_{0}^{t}{\ltwonorm[\Gamma]{v(\tau)-v_h(\tau)}^2 \;d\tau} \\
    &\begin{multlined}[t]\lesssim \ltwonorm{v_h(0)-v(0)}^2 + \ltwonorm{\nabla u_h(0) - \nabla u(0)}^2  
      + T \int_{0}^{t}{\norm{\dot{\rho}(\tau)}^2_{\mathcal{X} } d\tau}\\
      +T \int_{0}^{t}{  \norm{\rho(\tau)}^2_{L^2(\R^d \setminus \Gamma)\times \Hpglobal{1}} 
      + \beta^{-1} \ltwonorm[\Gamma]{\theta(\tau)}^2 \;d\tau}.\end{multlined}
  \end{align}   
  The implied constant depends only on the stabilization parameter $\alpha$ from Definition \ref{def:ritz_projector}.

  If the operator $g: H^{1/2}(\Gamma) \to L^2(\Gamma)$ is continuous 
  (see Lemma~\ref{lemma:approx_nonlinearity}~(\ref{it:approx_nonlinearity_convergence_strong}) for a sufficient
  condition), then the right-hand side converges to $0$ for $h \to 0$.
\end{theorem}
\begin{proof}
The additional error function 
\begin{align*}
  e(t):=\colvec{e_u(t) \\ e_v(t)}:=\colvec{u_h - \Pi_1 u \\v_h - \Pi_2 v}
\end{align*}
 solves 
\begin{align*}
  \ltwoproduct{\nabla \dot{e}_u }{\nabla w_h} &= \ltwoproduct{\nabla e_v}{\nabla w_h} 
  -\ltwoproduct{\nabla \rho_v}{\nabla w_h} + \ltwoproduct{\nabla \dot{\rho}_u}{\nabla w_h}\\
  \ltwoproduct{\dot{e}_v}{z_h}&
                                \begin{multlined}[t]=-\ltwoproduct{\nabla e_u}{z_h} 
                                  - \dualproduct{g(\tracejump{v_h})-g(\tracejump{\Pi_2 v})}{\tracejump{z_h}}  \\
                                  +\ltwoproduct{\dot{\rho}_v}{z_h} + \dualproduct{\theta}{\tracejump{z_h}}
                                  + \alpha \ltwoproduct{\rho_{u}}{z_h},\end{multlined}
\end{align*}
for all $w_h,z_h \in \HH$.
From the strict monotonicity of $g$, we obtain by testing with $e$
\begin{align*}
  \frac{1}{2} \frac{d}{dt}\norm{e}^2_{\mathcal{X}}
  + \beta\norm{\tracejump{e_v}}_{L^2(\Gamma)}^2
  &\leq   
  \ltwonorm{\nabla  \rho_v}\ltwonorm{\nabla e_u} + \ltwonorm{\nabla  \dot{\rho}_u}\ltwonorm{\nabla e_u} \\
  &+\ltwonorm{\dot{\rho}_v}\ltwonorm{e_v} 
    + \ltwonorm[\Gamma]{\theta}\ltwonorm[\Gamma]{\tracejump{e_v}} \\
  &+ \alpha \ltwonorm{\rho_{u}}\ltwonorm{e_v}.   
\end{align*}
Young's inequality and integrating then gives
\begin{align*}
  \norm{e(t)}_{\mathcal{X}}^2
  + \beta\int_{0}^{t}{\norm{\tracejump{e_v}(\tau)}_{\Gamma}^2 \,d\tau}
  &\begin{multlined}[t]
    \lesssim  \norm{e(0)}_{\mathcal{X}}^2
    + T \int_{0}^{t}{\norm{\dot{\rho}(\tau)}_{\mathcal{X}}^2 \;d\tau}
    + \max{(1,\alpha^2)}\,T\, \int_{0}^{T}{\norm{\rho(\tau)}^2_{L^2(\R^d \setminus \Gamma)\times \Hpglobal{1}} \;d\tau} \\
  +T \beta^{-1} \int_{0}^{t}{\ltwonorm[\Gamma]{\theta(\tau)}^2 \;d\tau}
    + \frac{\beta}{T} \int_{0}^{t}{\ltwonorm[\Gamma]{\tracejump{e_v(\tau)}}^2 \; 
    + \norm{e(\tau)}^2_{\mathcal{X}} \;d\tau}.
  \end{multlined}
\end{align*}
By Gronwall's inequality,  with $\alpha$ dependence absorbed into the generic constant,
\begin{align*}
    \norm{e(t)}_{\mathcal{X}}^2
  + \beta\int_{0}^{t}{\norm{\tracejump{e_v}(\tau)}_{\Gamma}^2 \,d\tau}
   &\lesssim \norm{e(0)}_{\mathcal{X}}^2
      + T \int_{0}^{t}{\norm{\dot{\rho}(\tau)}^2_{\mathcal{X}} 
  +  \norm{\rho(\tau)}^2_{L^2(\R^d \setminus \Gamma) \times \Hpglobal{1}}
    + \beta^{-1} \ltwonorm[\Gamma]{\theta(\tau)}^2 \;d\tau} .
\end{align*}
By the triangle inequality $\norm{u - u_h}_{\mathcal{X}} \leq \norm{e}_{\mathcal{X}} + \norm{\rho}_{\mathcal{X}}$
this gives \eqref{eq:err_est_expl_projections}. In order to see convergence, we need to investigate
the different error contributions. Since we have $u,v \in L^{\infty}\left((0,T),H^{1}(\Omega)\right)$, we
get convergence of $\norm{\rho}_{\mathcal{X}} \to 0$.
We have  $\dot{v} \in L^{\infty}\left(0,T, L^2(\R^d \setminus \Gamma)\right)$ which implies convergence of $\dot{\rho}_v$,
  since we chose $\Pi_2$ as the projector from Lemma \ref{lemma:ltwo_stable_projection}. 
  Since $\dot{u} \in H^1(\R^d \setminus \Gamma)$, we also have $\dot{\rho}_u \to 0$ for $h \to 0$.
This means, as long as the nonlinear term converges, we obtain convergence of the fully discrete scheme.
\end{proof}

\begin{remark}
  It might seem advantageous to use Ritz projector $\Pi_1$ throughout the proof of
  Theorem~\ref{thm:full_convergence_low_regularity} as this choice eliminates the term $\norm{\nabla \rho_v}_{L^2(\R^d \setminus \Gamma)}$,
  but the Ritz projector is not defined for $\dot{v} \in L^2(\R^d)$. Thus we have to either assume
  additional regularity or use the projector $\Pi_2$. \eremk
\end{remark}

\subsection{Time discretization analysis}
\label{sect:analysis_td}
In order to estimate the full error due to the discretization of the boundary integral equation, we
  are interested in approximating the semigroup solution $u$ via a multistep method. This
problem has been studied in the literature and the following proposition gives a
summary of the results we will need.
\begin{proposition}
\label{prop:semigroup_approx}
  Let $\AA$ be a maximally monotone operator on a separable Hilbert
  space $H$ with domain $\domain(A) \subseteq H$.
  
  Let $u$ denote the solution from Proposition \ref{prop:komura_kato} and
  define the approximation sequence $u^n_{\Delta t}$ by
  \begin{align} 
    \label{eq:def_general_semigroup_approx}
    \frac{1}{\Delta t}\sum_{j=0}^k{\alpha_j u_{\Delta t}^{n-j}} &= \AA u_{\Delta t}^{n},
  \end{align}
  where we assumed $u^{j}_{\Delta t}=u(j\Delta t)$ for $j =  0,\dots k$.
  The $\alpha_j$ originate from the implicit Euler or the BDF2 method.

  Then $u_{\Delta t}^n$ is well-defined, 
  i.e.~\eqref{eq:def_general_semigroup_approx} has a unique sequence of solutions,
  with $u_{\Delta t}^n \in \domain(\AA)$.
  If $u_{\Delta t}^0,\dots,u^k_{\Delta t} \in \domain(\AA)$, then the following estimate holds for $N \Delta t \leq T$:
  \begin{align*}
    \max_{n=0,\dots,N}{\norm{u(t_n) - u_{\Delta t}^n}}
    &\leq C \norm{ \AA u_0} \left[ \Delta t + T^{1/2} (\Delta t)^{1/2}
      + \left(T+T^{1/2}\right)(\Delta t)^{1/3}\right].
  \end{align*}
  Assume that $u \in C^{p+1}([0,T],H)$, where $p$ is the order of the multistep method. Then
  \begin{align*}
    \max_{n=0,\dots,N}{\norm{u(t_n) - u_{\Delta t}^n}}
    &\leq C T \Delta t^{p}.
  \end{align*}
\end{proposition}
\begin{proof}
  The general convergence result was shown by Nevanlinna in
  \cite[Corollary 1]{nevanlinna}. The improved convergence rate is shown in 
  \cite[Chapter V.8, Theorem 8.2]{hairer_wanner_2} or follows directly by inserting the consistency error into
  the stability theorem \cite[Theorem 1]{nevanlinna}.
\end{proof}

\begin{remark}
  We will use a shifted version of the previous proposition, where we assume 
  $u^{-j}_{\Delta t}=u(-j\Delta t)$ for $j \in \N$ and define all $u^n_{\Delta t}$ via \eqref{eq:def_general_semigroup_approx}
  for all $n \in \N$. This does not impact the stated results. \eremk
\end{remark}

Now that we have specified the semigroup setting, we can
write down the multistep approximation sequence ${(\usg,\vsg) \subseteq \domain(\AA)}$ of
\eqref{eq:wave_eqn} via
\begin{subequations}
\label{eq:def_usg_vsg}
\begin{align*}
  \frac{1}{\Delta t} \sum_{j=0}^k{\alpha_j \usg^{n-j}} &=  \vsg^{n}, \\
  \frac{1}{\Delta t} \sum_{j=0}^k{\alpha_j \vsg^{n-j}} &=  \laplace \usg^{n},
\end{align*}
\end{subequations}
together with the initial conditions $\usg^j=u^{inc}(j\Delta t)$, $\vsg^j=\dot{u}^{inc}(j \Delta t)$ for
$j<0$; see Proposition~\ref{prop:semigroup_approx}.

Comparing this definition to (\ref{eq:diff_eq_u_ie}), we see that due to the way we dealt with $u^{inc}$, the approximation of the semigroup does not coincide with
the approximation induced by the boundary integral equations.
The following lemma shows that this error does not compromise the convergence rate.
\begin{lemma}
  \label{lemma:comp_uie_usg}
  Let $p > 0$ denote the order of the multistep method, and assume
  \begin{align*}
    \left(u_{inc},\dot{u}_{inc}\right) \in C^{\alpha}\left((0,T), \mathcal{X}\right)
  \end{align*}
  for $\alpha > 1$.
  We consider the shifted version of $\uie$, defined via 
  $\uiet:=\uie + \Pi_1 u^{inc}$
  and $\viet:=\dd \uie + \Pi_0\dot{u}^{inc}$;
  $\Pi_0$ and $\Pi_1$ are defined as in Section~\ref{sect:approx_in_hh}.
  Then the following error estimate holds:
  \begin{align*}
    \left(\ltwonorm{ \vsg^n -\viet^n }^2 + 
    \ltwonorm{\nabla \usg - \nabla \uiet}^2\right)^{1/2}
    \leq& C T \left(\Delta t\right)^{\min\left(p,\alpha-1\right)} 
    \norm{\left(u^{inc},\dot{u}^{inc}\right)}_{C^\alpha\left((0,T), \mathcal{X}\right)} \\
    &+C \Delta t 
    \sum_{j=0}^{n}{{\norm{ \left(I - \bdryinterpY\right)\gamma^+\dot{u}^{inc}(t_n)  }_{H^{1/2}(\Gamma)}}} \\
    &+C \Delta t \sum_{j=0}^{n}{\inf_{x_h \in X_h}{\norm{ \partial_n^+ \dot{u}^{inc}(t_n) - x_h }_{H^{-1/2}(\Gamma)}}} \\
    &+ C \Delta t
      \sum_{j=0}^{n}{\norm{ \left(I-\bdryinterpY\right) \gamma^+  \ddot{u}^{inc}(t_n) }_{H^{1/2}(\Gamma)}}.      
  \end{align*}
The same convergence rates hold  if we use $u^{inc}$, $\dot{u}^{inc}$ instead of the projected versions on the
  left hand side. Thus in practice we do not depend on the
  non-computable operators $\Pi_0$ and $\Pi_1$.
\end{lemma}
\begin{proof}
  Inserting the definition of $\uiet$ and $\viet$ into the multistep method, 
  using \eqref{eq:diff_eq_u_ie}, we see that $(\uiet,\viet)$ solves
  \begin{align*}
    \frac{1}{\Delta t}\sum_{j=0}^{k}{\alpha_j \uiet^{n-j}} &= \viet^{n} + \varepsilon^n, \\
    \frac{1}{\Delta t} \sum_{j=0}^{k}{\beta_j \viet^{n-j}} &= \laplace \uiet^{n} + \theta^n,
  \end{align*}
  with right-hand sides
  \begin{align*}
    \varepsilon^n&:=\frac{1}{\Delta t}
                   \sum_{j=0}^{k}\alpha_j \Pi_1 u^{inc}\left(t_{n-j}\right)
                   -\Pi_0 \dot{u}^{inc}(t_{n}),\\
    \theta^n&:=\frac{1}{\Delta t}
    \sum_{j=0}^{k}{\alpha_j \Pi_0 \dot{u}^{inc}\left(t_{n-j}\right)}
            -  \laplace \Pi_1 u^{inc}(t_{n}).
  \end{align*}
  We have shown that $\AA$ is maximally monotone and that from the properties
  of $\Pi_0$ and $\Pi_1$ we have $(\uiet^n,\viet^n) \in \domain(\AA)$.
  Thus, we can apply \cite[Theorem 1]{nevanlinna} to get for the difference $\uiet-\usg$
  \begin{align*}
    \left(\ltwonorm{  \vsg^n - \viet^n}^{2}
        + \ltwonorm{\nabla \uiet^n - \nabla \usg^n}^{2}\right)^{1/2}
    &\leq \Delta t \sum_{j=0}^{n}{\left(\ltwonorm{\theta^n}^{2} + \ltwonorm{\nabla \varepsilon^n}^{2}\right)^{1/2}}.
  \end{align*}
  It remains to estimate the error terms $\ltwonorm{\theta^n}$ and $\ltwonorm{\nabla \varepsilon^n}$.
  We start with $\varepsilon^n$ and rewrite it as
  \begin{align*}
    \varepsilon^n&=\Pi_1 \left( \frac{1}{\Delta t}\sum_{j=0}^{k}{\alpha_j  u^{inc}\left(t_{n-j}\right)
                 - \dot{u}^{inc}(t_{n})} \right)
                 +(\Pi_0 - \Pi_1)\left(  \dot{u}^{inc}(t_{n}) \right).
  \end{align*}
  For the norms, this gives due to the stability of $\Pi_1$ and the approximation properties of
  $\Pi_1$ and $\Pi_0$
  \begin{align*}
    \ltwonorm{\nabla \varepsilon^n}
    &\lesssim \Hpnorm{1}{ \frac{1}{\Delta t}\sum_{j=0}^{k}{\alpha_j  u^{inc}\left(t_{n-j}\right)
      - \dot{u}^{inc}(t_{n})}  }  
    +{\norm{ \left(I - \bdryinterpY\right)\dot{u}^{inc}(t_{n}) }_{H^{1/2}(\Gamma)}}
      +\inf_{x_h \in X_h}{\norm{ \partial_n^+ \dot{u}^{inc}(t_{n}) - x_h }_{H^{-1/2}(\Gamma)}}
  \end{align*}
  The first term is $\bigO\left(\Delta t^{{\min\left(p,\alpha-1\right)}} \right)$ as the consistency
  error of a $p$-th order multistep method. 

  A similar argument can be employed for $\tau_n$.
 Noticing that $\laplace \Pi_1 u= \laplace u$ and arranging the terms as above gives
  \begin{align*}
    \ltwonorm{\theta^n}
    &\lesssim \ltwonorm{ \frac{1}{\Delta t}\sum_{j=0}^{k}{\alpha_j  \dot{u}^{inc}\left(t_{n-j}\right)
      -  \laplace u^{inc}(t_{n})} } 
    + \ltwonorm{ (I -\Pi_0)  \laplace u^{inc}(t_{n})}.
  \end{align*}
  Since we assumed that $u^{inc}$ solves the wave equation, we have that
  \begin{align*}
    \ltwonorm{\theta^n}&\leq \bigO\left(\left(\Delta t\right)^{{\min\left(p,\alpha-1\right)}} \right) 
    +C \sum_{j=0}^{k}{\ltwonorm{ (I -\Pi_0)  \ddot{u}^{inc}(t_n)}},    
  \end{align*}
  which concludes the proof.
\end{proof}

The following convergence result with respect to the time-discretization now follows.
\begin{theorem}
\label{thm:time_discretization}
Assume for a moment that $X_h=H^{-1/2}$ and $Y_h=H^{1/2}(\Gamma)$.
The discrete solutions, obtained by $\uie:=S(\dd)\varphi + (\dd)^{-1}D(\dd) \psi$ converge to the
exact solution $u$ of \eqref{eq:ode_semigroup} with the rate
\begin{align*}
    \max_{n=0,\dots,N}{\norm{u(t_n) - \uie(t_n) - u^{\text{inc}}(t_n)}}
     &\lesssim T (\Delta t)^{1/3}.
\end{align*}
If we assume that the exact solution satisfies
    $(u,\dot{u}) \in C^{p+1}\left([0,T], BL^1 \times \ltwoglobal \right)$, then we regain the
full convergence rate of the multistep method
\begin{align*}
  \max_{n=0,\dots,N}{\norm{u(t_n) - \uie(t_n) - u^{\text{inc}}(t_n)}}
  &\lesssim T (\Delta t)^{p}.  
\end{align*}
For the fully discrete setting, the same rates in time hold, but with additional 
projection errors due to $u^{inc}$; see Lemma~\ref{lemma:comp_uie_usg}.
\end{theorem}
\begin{proof}
  This statement is easily obtained by combining Proposition \ref{prop:semigroup_approx} with
  Lemma \ref{lemma:comp_uie_usg}. 
\end{proof}

\subsection{Analysis of the fully discrete scheme}
\label{sect:analysis_fully_discrete}
In this section, we can now combine the ideas and results from the previous sections, to get convergence  results for the fully discrete approximation to
the semigroup, and therefore also to the approximation obtained by solving the boundary integral equations (\ref{eq:fully_discrete_int_eq}) and using the representation formula.

We start by combining the estimates from Sections~\ref{sect:sd_in_space} and ~\ref{sect:analysis_td},
which immediately give a convergence result for the full discretization:

\begin{corollary}
\label{thm:expl_convergence_statement_low_regularity}
  Assume that the incoming wave satisfies $(u^{\text{inc}},\dot{u}^{\text{inc}}) \in C^{\alpha}\left((0,T), \mathcal{X}\right)$ for $\alpha > 1$.
  Setting $\widetilde{u}_{ie}^{\Delta t}(t_n):=\uie^{\Delta t} + u^{\text{inc}}(t_n)$ and $\widetilde{v}_{ie}^{\Delta t}(t_n):=\vie^{\Delta t} + \dot{u}^{\text{inc}}(t_n)$
  the discretization error can be quantified by:
  \begin{align*}
    \norm{\widetilde{u}_{ie}^{\Delta t}(t_n)  - u(t_n)}_{BL^1} + \norm{\widetilde{v}_{ie}^{\Delta t}(t_n)  - \dot{u}(t_n)}_{L^2(\R^d)}
    \lesssim& \left(\norm{\mathcal{A} u(0)}_{\mathcal{X}}+\norm{(u,\dot{u})}_{C^{\alpha}((0,T),\mathcal{X})}\right)  T \left(\Delta t  \right)^{\min(\alpha-1,1/3)}   \\
    &+ \Delta t \sum_{j=0}^{n}{{\norm{ \left(I - \bdryinterpY\right)\gamma^+\dot{u}^{inc}(t_n)  }_{H^{1/2}(\Gamma)}}}  \\
    &+ \Delta t \sum_{j=0}^{n}{\inf_{x_h \in X_h}{\norm{ \partial_n^+ \dot{u}^{inc}(t_n) - x_h }_{H^{-1/2}(\Gamma)}}} \\    
    &+ \Delta t \sum_{j=0}^{n}{\norm{ \left(I-\bdryinterpY\right) \gamma^+  \ddot{u}^{inc}(t_n) }_{H^{1/2}(\Gamma)}}. \\              
    & + T \int_{0}^{t_n}{\norm{\dot{\rho}(\tau)}^2_{\mathcal{X} } 
   + \norm{\rho(\tau)}^2_{L^2(\R^d \setminus \Gamma)\times \Hpglobal{1}} 
     + \beta^{-1} \ltwonorm[\Gamma]{\theta(\tau)}^2 \;d\tau}.      
  \end{align*}
with the error terms
\begin{align*}
  \rho(t)&:=\colvec{\rho_u(t) \\ \rho_v(t)}:=\colvec{u - \Pi_1 u \\v - \Pi_2 v}, \\
  \theta(t)&:=g(\tracejump{v})-g(\tracejump{\Pi_2 v}).
\end{align*}
  Assuming $\abs{g(\mu)}\lesssim \left(1+\abs{\mu}^{\frac{d-1}{d-2}}\right)$ for
  $d\geq 3$, this this gives strong convergence
  \begin{align*}
    \uie^{\Delta t} + u^{inc} &\to u  \quad \text{in } L^{\infty}\left((0,T); BL^1\right), \\
    \dd \uie^{\Delta t} + \dot{u}^{inc} & \to \dot{u} \quad \text{in } L^{\infty}\left((0,T);L^2(\R^{d})\right),
  \end{align*}
  with a rate in time of $(\Delta t)^{1/3}$ and quasi-optimality in space.
\end{corollary}
\begin{proof}
  We just collect all the estimates from the previous sections. The stronger growth condition on $g$ is needed to ensure that nonlinearity-error $\theta$ converges in $L^2(\Gamma)$
  (see Lemma~\ref{lemma:approx_nonlinearity}~(\ref{it:approx_nonlinearity_convergence_strong}))
\end{proof}

\subsubsection{The case of higher regularity}
\label{subsect:high_regularity}
Previously, we avoided assuming any regularity of the exact solution $u$.
In order to get the full convergence rates, we need to make these additional assumptions.

\begin{assumption}
\label{assumption:u_is_regular}
  Assume that the exact solution of \eqref{eq:wave_eqn} has the following regularity properties:
  \begin{enumerate}[(i)]
    \item $\displaystyle u \in C^{p+1}\left((0,T); H^1(\Omega^{+})\right)$,
    \item $\displaystyle \dot{u} \in C^{p+1}\left((0,T); L^2(\Omega^{+})\right)$,
    \item $\gamma^+u,\gamma^+ \dot{u} \in L^{\infty}\left((0,T), H^{m}(\Gamma)\right)$,
    \item $\partial_n^+u, \partial_n^+\dot{u} \in L^{\infty}\left((0,T), H^{m-1}(\Gamma)\right)$,
    \item $\ddot{u} \in L^{\infty}\left((0,T), H^{m}(\Omega^-)\right)$,    
    \item $\gamma^+ \dot{u} \in L^{\infty}\left((0,T) \times \Gamma\right)$,
  \end{enumerate}
  for some $m\geq 1/2$. Here $p$ denotes the order of the multistep method that is used.
\end{assumption}

\begin{remark}
  We need the strong requirement of 
  $\gamma^+ \dot{u} \in L^{\infty}\left((0,T) \times \Gamma\right)$ in order to be able to apply
  Lemma~\ref{lemma:approx_nonlinearity}~(\ref{it:approx_nonlinearity_est_with_l_infty}). If we make stronger
  growth assumptions on $g'$, we can drop this requirement. \eremk
\end{remark}

Since we only made assumptions on the exact solution $u(t)$, we cannot
use the same procedure as for Theorem \ref{thm:full_convergence_low_regularity} of first considering semidiscretization in space and then treating the discretization in time separately,  since
this would require knowledge of the regularity of $u_h(t)$. 
Instead, we will perform the discretization in space and time simultaneously
using a variation of Theorem~\ref{thm:full_convergence_low_regularity}
in the time discrete setting.

The $G$-stability of the linear multistep methods used  allows us to analyse the convergence rate of the method, given that the exact solution has sufficient regularity in space and time.

\begin{lemma}
\label{lemma:smooth_convergence}
  Assume, that $g$ is strictly monotone as defined in~\eqref{def:strictly_monotone}.
  Let $\usgh^{n}$,$\vsgh^n$ denote the sequence of 
  approximations obtained by applying the BDF1 or BDF2 method to \eqref{eq:ode_semigroup}
  in $\HH$ and $u$ the exact solution of~\eqref{eq:wave_eqn}, with $v:=\dot{u}$.
  We will use the finite difference operator $D^{\Delta t}: \ell^2 \to \ell^2:$
  $[D^{\Delta t}u]^n:=\frac{1}{\Delta t}\sum_{j=0}^{k}{\alpha_j u^{n-j}}$ and with the same notation for continuous $u$ set
    $[D^{\Delta t}u]^n:=\frac{1}{\Delta t}\sum_{j=0}^{k}{\alpha_j u(t_{n-j})}$.

  We introduce the following error terms:
  \begin{align*}
    \Theta_I^n&:=\Pi_1\left( [D^{\Delta t} u]^n - \dot{u}(t_n) \right) \\
    \Theta_{II}^n&:=\Pi_2\left( [D^{\Delta t}v]^n - \dot{v}(t_n) \right) \\
    \Theta_{III}^n&:=\left(\Pi_1 - \Pi_2\right) v(t_n) \\
    \Theta_{IV}^n&:=\left(I - \Pi_2\right) \dot{v}(t_n) \\
    \Theta_{V}^n&:=g(\tracejump{v(t_n)})-g(\tracejump{\Pi_2 v(t_n)}) \\
    \Theta_{VI}^n&:=\left(I-\Pi_1\right)u(t_n).    
  \end{align*}
  Then 
  \begin{align}
    \label{eq:err_est_expl_projections_fd}
    \ltwonorm{\vsgh^n - v(t_n)}^2 &+ \ltwonorm{\nabla \usgh^n - \nabla u(t_n)}^2 +
      \beta \Delta t \sum_{j=0}^{n}{\ltwonorm[\Gamma]{\tracejump{(\vsgh^j - v(t_j))}}^2 \;ds} \\
    \lesssim \Delta t \sum_{j=0}^{n}&{ 
    \ltwonorm{\nabla \Theta_I^j}^2}
    + \ltwonorm{\Theta_{II}^j}^2 + \ltwonorm{\nabla{\Theta_{III}^j}}^2 + \ltwonorm{\Theta_{IV}^j}^2 \\
    &+\Delta t \sum_{j=0}^{n}{ \ltwonorm[\Gamma]{\Theta_{V}^j}^2  + \norm{\Theta_{VI}^j}_{\ltwoglobal}^2}.
  \end{align}
  The implied constant depends only on the parameter $\alpha$ in the definition of $\Pi_1$.
\end{lemma}
\begin{proof}
The proof is fairly similar to the one of Lemma~\ref{lemma:comp_uie_usg} and many of the similar
terms  appear.
We define the additional error sequence
\begin{align*}
  e^n:=\colvec{e_u^n \\ e_v^n}:=\colvec{\Pi_1 u(t_n) - \usgh^n \\ \Pi_2 v(t_n) -\vsgh^n }.
\end{align*}
The overall strategy of the proof is to substitute $e$ in the defining
equation for the multistep method and compute the truncation terms. We then test with $e^n$
in order to get discrete stability just as we did in 
Theorem~\ref{thm:expl_convergence_statement_low_regularity}.
The proof becomes technical, due to the many different error terms which appear.

The error $e^n$ solves the following equation for all $w_h,z_h \in \HH$, see
\eqref{eq:weak_form_sg_space_c}, 
\begin{align*}
  \ltwoproduct{\nabla [D^{\Delta t} e_u]^n }{\nabla w_h} &= \ltwoproduct{\nabla e_v^n}{\nabla w_h} +
                                         \ltwoproduct{\nabla \left(\Theta_I + \Theta_{III}\right)}{\nabla w_h}\\
  \ltwoproduct{[D^{\Delta t} e_v]^n}{z_h}&\begin{multlined}[t]=-\ltwoproduct{\nabla e_u}{\nabla z_h} 
  - \dualproduct{g(\tracejump{\vsgh^n})-g(\tracejump{\Pi_2 v(t_n)})}{\tracejump{z_h}}  \vspace{1mm}\\
+ \ltwoproduct{\Theta_{II} + \Theta_{IV} + \alpha \, \Theta_{VI}}{z_h} 
+ \dualproduct{\Theta_V}{\tracejump{z_h}}.\end{multlined}
\end{align*}

From the strict monotonicity of $g$, we obtain by testing with $e^n$ that
\begin{align*}
\dualproduct[\mathcal{X}]{[ D^{\Delta t} e]^n}{e^n }  
  + \beta\norm{\tracejump{e_v}}_{L^2(\Gamma)}^2
  &\leq\begin{multlined}[t]
    \ltwoproduct{\nabla \left( \Theta_I + \Theta_{III} \right)}{\nabla e^n_u} 
    + \dualproduct{\Theta_V}{\tracejump{e_v^n}} \vspace{1mm}\\
  + \ltwoproduct{\Theta_{II} + \Theta_{IV} 
  + \alpha \, \Theta_{VI}}{e^n_v}.  \end{multlined}
\end{align*}
We write $E^n:=(e^{n},\dots,e^{n-k})^T$  and use Proposition~\ref{prop:Gstability} to obtain a lower bound on the left-hand side. 
Using the Cauchy-Schwarz and Young inequalities on the right-hand side then gives
\begin{align*}
  \norm{E^n}_{G}^2 - \norm{E^{n-1}}^2_{G}
  + \beta \Delta t \norm{\tracejump{e_v^n}}_{L^2(\Gamma)}^2
  &\begin{multlined}[t]\lesssim  
    \Delta t \ltwonorm{\nabla \Theta_I + \nabla \Theta_{III}}^2
    + \Delta t \ltwonorm[\Gamma]{\Theta_V}^2 \\
  + \Delta t \ltwonorm{\Theta_{II} + \Theta_{IV} 
  + \alpha \, \Theta_{VI}}^2.\end{multlined}
\end{align*}
Summing over $n$ and noting the equivalence of the $G$-induced-norm to the standard $\R^{k}$ norm  gives the stated result.

\end{proof}

\begin{assumption}
\label{assumption:approximation_spaces}
  Assume that the spaces $X_h$, $Y_h$ and the operator $\bdryinterpY$ satisfy the following approximation properties
\begin{subequations}
    \label{eq:approx_spaces_assumption}
    \begin{align}
      \inf_{x_h \in X_h} \norm{ \varphi - x_h}_{H^{-1/2}(\Gamma)}
      &\leq C h^{1/2+\min(m,q+1)}
        \norm{\varphi}_{H^{m}_{pw}(\Gamma)}  \quad \quad \forall \varphi \in H^{m}_{pw}(\Gamma)), \\
        \norm{ \psi - \bdryinterpY \psi}_{H^{1/2}(\Gamma)}
      &\leq C h^{\min(m,q+1)-1/2}
        \norm{\psi}_{H^{m}_{pw}(\Gamma)}  \quad \quad \forall \psi \in H^{m}_{pw}(\Gamma), \label{eq:approx_spaces_assumption_yh}
    \end{align}
  \end{subequations}
  for  parameters $h>0$ and $q \in \N$, with constants that depend only on $\Gamma$ and $q$.
  Assume further  that the fictitious space $\widetilde{Y}_h$ and the operator $\volumeinterpY$ from Assumption~\ref{ass:X_tilde} also
  satisfy
  \begin{align}
    \label{eq:approx_spaces_assumption_l2}
    \norm{ u - \volumeinterpY{u}}_{L^2(\Omega^-)}
    &\leq C h^{\min(q+1,m)}
        \norm{u}_{H^{m}(\Omega^-)}  \quad \quad \forall u \in H^{m}(\Omega^-). 
  \end{align}
\end{assumption}

\begin{theorem}
  Assume that Assumptions~\ref{assumption:u_is_regular} and \ref{assumption:approximation_spaces} are
  satisfied and that we use BDF1($p=1$) or BDF2($p=2$) discretization in time. 
  Assume either $\abs{g'(s)}\lesssim 1+\abs{s}^q$ with $q \in \N$ for $d=2$ or $q\leq 1$ for $d=3$,
    or assume that $\norm{\Pi_2 \dot{u}}_{L^{\infty}(\R^d)} \leq C$ is uniformly bounded w.r.t.\ $h$.
  Then the following convergence result holds
  \begin{align*}
    \ltwonorm{\vsg^n - \dot{u}(t_n)}^2 + \ltwonorm{\nabla \usgh^n - u(t_n)}
    &=\bigO\left(T \left( h^{\min(q+1/2,m)} + \Delta t^p \right)\right).
  \end{align*}
  The constants involved depend only on $\Gamma$, $g$, $\alpha$, 
  and the constants implied in Assumptions \ref{assumption:u_is_regular}
  and Assumption \ref{assumption:approximation_spaces}.
\end{theorem}
\begin{proof}
  We already have all the ingredients for the proof.
  We combine Theorem~\ref{lemma:smooth_convergence} with the approximation properties of the operators
  from Section~\ref{sect:approx_in_hh}. $\Theta_I$ and $\Theta_{II}$ are the local truncation errors of the
  multistep method and therefore $\bigO(\Delta t)$ (the operators $\Pi_1$, $\Pi_2$ are stable). To estimate $\Theta_{V}$,
  the assumptions on $g$ or $\Pi_2$ are such that we can apply Lemma~\ref{lemma:approx_nonlinearity}.  
\end{proof}

\begin{remark}
  The assumptions on the spaces $X_h$,$Y_h$ are true if we use standard piecewise polynomials of
  degree $p$ for $Y_h$ and $p-1$ for $X_h$ on some triangulation $\TT_h$ of $\Gamma$.
  See \cite[Theorem 4.1.51, page 217]{book_sauter_schwab} and \cite[Theorem 4.3.20, page 260]{book_sauter_schwab}  
  for the proofs of the approximation properties.
  In this case, the approximation property \eqref{eq:approx_spaces_assumption_l2} holds via the
  construction from Lemma~\ref{lemma:ltwo_stable_projection} using the Scott-Zhang projection
  and standard finite element theory. The requirements on $\Pi_2 \dot{u}$ can be fulfilled in numerous ways,
  e.g., in 2D and 3D it can be shown using standard approximation and inverse estimates  as in \cite[Lemma 13.3]{thomee_book}.
  The same result can also be achieved by replacing $\Pi_2$ by some projector
  that allows $L^{\infty}$-estimates, e.g., nodal interpolation.\eremk
\end{remark}

\subsection{The case of more general $g$}
\label{sect:general_g}
In  all the previous theorems we assumed that $g$ was strictly monotone. In this section, we sketch what happens
if we drop this requirement. Most notably we lose all explicit error bounds and also the strong convergence.  What can be salvaged is 
a weaker convergence result.

We start with the weaker version of Theorem~\ref{thm:full_convergence_low_regularity} telling us that the semidiscretization with regards to space
converges weakly.
\begin{lemma}  
  \label{lemma:convergence_space_general_g}
  Assume the families of spaces $(X_h)_{h > 0}$ and $(Y_h)_{h>0}$ are dense in $H^{-1/2}(\Gamma)$ and
  $H^{1/2}(\Gamma)$ respectively.
  Then the sequence of solutions
  $u_h(t),v_h(t)$ of \eqref{eq:weak_form_sg_space_d} weakly towards the solution of \eqref{eq:weak_form_sg_space_c} for almost all $t \in (0,T)$
  for $h \to 0$. 
\end{lemma}
\begin{proof}
  We fix $t \in (0,T]$; all the arguments hold only almost everywhere w.r.t $t$, but
  that is sufficient in order to prove the result.
  Since $g(\tracejump{v_h})\tracejump{v_h}\geq 0$,  testing with $w_h=u_h$ and $z_h=v_h$ in \eqref{eq:weak_form_sg_space_d} gives
  \begin{align*}
    \norm{u_h(t),v_h(t)}_{\mathcal{X}} &\leq \norm{u_h(0),v_h(0)}_{\mathcal{X}}.
  \end{align*}
  By the Eberlein-\v{S}mulian theorem, see for example \cite[page 141]{yosida_fana}, this gives a weakly convergent sub-sequence, that we 
  again denote by $u_h,v_h$ --- uniqueness of the solution will give convergence of the whole sequence anyway --- and write $u,v$ for its weak limit.
  It is easy to see that $\dot{u}_h \rightharpoonup \dot{u}$ and $\dot{v}_h \rightharpoonup \dot{v}$
   since the convergence is only with respect to the spatial discretization.  
  
  From the estimate in Proposition~\ref{prop:komura_kato}, we get
  $\norm{\dot{u}_h(t),\dot{v}_h(t)}_{\mathcal{X}}\leq C\left(u^{inc}\right)$,
  since $\left(u^{inc}(0),v^{inc}(0) \right) \in \HH$ as the incoming wave vanishes
  at the scatterer for $t=0$. Since $\dot{u}_h=v_h$,
  we have that $\norm{v_h(t)}_{\Hpglobal{1}}$ is uniformly bounded.
  This implies, up to a sub-sequence,
  that the trace also converges: $\tracejump{v_h(t)} \rightharpoonup \tracejump{v(t)}$ in $H^{1/2}(\Gamma)$.
  What remains to show is that
  $g\left(\tracejump{v_h(t)}\right) \rightharpoonup g\left(\tracejump{v(t)}\right)$.
  This was already done in Lemma~\ref{lemma:approx_nonlinearity}~(\ref{it:approx_nonlinearity_convergence_weak}).
\end{proof}

The convergence of the time-discretization does not depend on the strong monotonicity of $g$. This insight immediately gives the following corollary:
\begin{corollary}
  Assume the families of spaces $(X_h)_{h > 0}$ and $(Y_h)_{h>0}$ are dense in $H^{-1/2}(\Gamma)$ and
  $H^{1/2}(\Gamma)$ respectively and assume that the operator $\bdryinterpY$ converges strongly,
    i.e. for $y \in H^{1/2}(\Gamma)$,  $\bdryinterpY y \to y$ converges when $h \to 0$.
 
  Let $\uiet^{\Delta t}$, $\viet^{\Delta t}$ denote the piecewise linear interpolant between the time-stepping approximations $\uiet^n$, $\viet^n$
  from Lemma~\ref{lemma:comp_uie_usg} at nodes $n \Delta t$. Then these approximations converge weakly towards the solution of \eqref{eq:wave_eqn} for almost all $t \in (0,T)$
  for $\Delta t \to 0$ and $h \to 0$.  
\end{corollary}
\begin{proof}
  Inspecting the proof of Lemma~\ref{lemma:comp_uie_usg} and Theorem~\ref{thm:time_discretization}, we didn't require $g$ to be strictly monotone, therefore we get that
  $\uiet^{\Delta t}$ and $\viet^{\Delta t}$ converge (strongly) to $u_h$ and $v_h$ respectively. Then the result follows directly from Lemma~\ref{lemma:convergence_space_general_g}.
\end{proof}

\section{Numerical Results}
\label{sect:numerics}
In this section, we investigate the convergence of the numerical method via numerical experiments.
We implemented the algorithm using the BEM++ software library \cite{bempp}. 
In order to solve the nonlinear equation~\eqref{eq:fully_discrete_int_eq}, we linearize $g$ and
performed a Newton iteration, i.e.\ at each Newton step, we solve
\begin{align}
\label{eq:linearized_integral_eqn}
  \dualproduct{ B_{\text{imp}} \left(\frac{\delta(0)}{\Delta t}\right)
    \colvec{\varphi^{n,k+1} \\ \psi^{n,k+1}}} {\colvec{\xi \\ \eta}}
  + \dualproduct{g'(\psi^{n,k} + \bdryinterpY \dot{u}^{inc}(t_n)) \psi^{n,k+1} }{\eta}
&= \dualproduct{f^n}{\colvec{\xi \\ \eta}}
  + \dualproduct{g_n^{k}}{\eta},
\end{align}
 with $\displaystyle f^n:=-\colvec{0 \\ \partial_n^+ u^{inc}(t_n)} - \sum_{j=0}^{n-1}{ B_{n-j} \colvec{\varphi^{j} \\\psi^j}}$
 and $g_n^{k}:=-g\left(\psi^{n,k} + \bdryinterpY \dot{u}^{inc}(t_n)\right) 
 + g'\left(\psi^{n,k} +\bdryinterpY \dot{u}^{inc}(t_n) \right) \psi^{n,k}$.
 
 For $\bdryinterpY$ we used an $L^2$ projection.
 As we can see, the right-hand side consists of two parts, where $g_n^k$ has to be recomputed in each
 Newton step but only involves local computations in time and space. The computationally much more expensive
 part $f^n$ is the same for each Newton step and thus has only to be computed once.
 Therefore,  as long as the convergence of the Newton iteration is reasonably fast, the additional cost due to the nonlinearity is small. This was already noted in
 \cite{banjai_waves2015}.  In order to efficiently solve these convolution equations, we employed the recursive algorithm
 based on approximating the convolution weights with an FFT described in\cite{banjai_algorithms}.

 For $X_h$ we used piecewise polynomials of some fixed degree $p$ on a triangulation of $\Gamma$ and for $Y_h$
  globally continuous piecewise polynomials of degree $p+1$.
 
 In order to be able to compute an estimate of the error of the numerical method, we need
 to have a good approximation to the exact solution. This was obtained by choosing
 a sufficiently small step size compared to the numerical approximation and always use the second order BDF2-method; we used
 at least $\Delta t_{ex} \leq \frac{\Delta t}{4}$ for the scalar examples and $\Delta t_{ex} \leq \frac{\Delta t}{2}$ for the full 3D problems.

 Since it is difficult to compute the norms $\norm{u}_{\Hpglobal{1}}$ and $\ltwonorm{v}$ from
   the representation formula, we will instead compare the errors of the traces on the boundary. The
   convergence rate of these is the content of Lemma~\ref{lemma:error_of_traces}. To prove this lemma, we
   first need the following simple result.
 \begin{lemma}
   \label{lemma:error_est_after_integration}
   Let $f \in C^{r}(X)$, $\widetilde{f}\in C(X)$ for some Banach space $X$,
   with $0=f(0)=f'(0)=\dots=f^{(r-1)}(0)=\widetilde{f}(0)$ and $r\leq p$. Then 
   \begin{align*}
     \norm{ \partial_t^{-1} f  - \left(\dd\right)^{-1} \widetilde{f} }_{X}
     &\leq C t \left[ \left(\Delta t\right)^{r} \max_{\tau\in[0,t]}{\norm{f^{r}(\tau)}_{X} } 
       +  \max_{\tau \in [0,t]}{\norm{f(\tau) - \widetilde{f}(\tau)}_{X}}\right].       
   \end{align*}
 \end{lemma}
 \begin{proof}
   We split the error into two terms by writing
   \begin{align*}
     \norm{ \partial_t^{-1}f  - \left(\dd\right)^{-1} \widetilde{f} }_{X}
     &\leq \norm{ \partial_t^{-1}f  - \left(\dd\right)^{-1} f }_{X} +
       \norm{ \left(\dd\right)^{-1} \left( f - \widetilde{f} \right) }_{X}.
   \end{align*}
   The stated estimate then follows from the standard theory of convolution quadrature;  see \cite[Theorem 3.1]{lubich_cq1},  noting that $\partial_t^{-1}$ is a sectorial operator.
 \end{proof}

 This now allows us to prove convergence estimates for $\psi$ and $\varphi$.
 \begin{lemma}
   \label{lemma:error_of_traces}
   Let $u$ solve \eqref{eq:ode_semigroup}, write $\widetilde{u}(t):=u(t) - u^{inc}(t)$,  and define the traces $\psi(t):=\partial_t \gamma^+ \widetilde{u}(t)$ and
   $\varphi(t):=-\partial_n^+ \widetilde{u}(t)$.
   Let $\uie$, $\vie$ solve \eqref{eq:diff_eq_u_ie}, with corresponding traces
   $\varphi_{ie}:=-\normaljump{\uie}$, $\psi_{ie}:=\tracejump{\vie}$.
   Then 
   \begin{align*}
     \triplenorm{\colvec{\partial_t^{-1} \psi(t_n) - \left[\left(\dd\right)^{-1} \psi_{ie}\right]^n \vspace{1mm}\\ 
     \partial_t^{-1} \varphi(t_n) - \left[\left(\dd\right)^{-1} \varphi_{ie}\right]^n }}    
     &\lesssim \ltwonorm{ \nabla \widetilde{u}(t_j) - \nabla \uie^{j}} + 
       \ltwonorm{\dot{\widetilde u}(t_j) - \vie^{j}}  + \bigO\left(\Delta t^{r}\right),
   \end{align*}
holds  for all $n$ with $n \Delta t \leq T$, with constants
   that depend only on $\Gamma$ and the end time $T$, and where $r$ is the minimal regularity  index of $\psi(t)$ and $\varphi(t)$.
 \end{lemma}
 \begin{proof}
   From the definition of $\psi$ and
   the properties of the operational calculus we have that
   \begin{align*}
      \left(\dd\right)^{-1} \psi_{ie}&=\left(\dd\right)^{-1}\tracejump{\vie} 
                                    =\left(\dd\right)^{-1} \dd \tracejump{\uie}  = \tracejump{\uie}
   \end{align*}
   and analogously $\partial_t^{-1} \psi= \tracejump{\widetilde{u}}$.
   The standard trace theorem then gives the estimate
   \begin{align*}
     \norm{\partial_t^{-1} \psi(t_n) - \left[\left(\dd\right)^{-1} \psi_{ie}\right]^n}_{H^{1/2}(\Gamma)}
     &\leq C \norm{ \widetilde{u} - \uie}_{\Hpglobal{1}}.
   \end{align*}
   We can further estimate the $L^2$ contribution in the norm above by noting
   \begin{align*}
     \ltwonorm{\widetilde{u}(t_n) - \uie^n}
     &= \ltwonorm{\partial^{-1}_t \dot{\widetilde{u}}(t_n) - \left(\dd\right)^{-1} \vie^n}
   \end{align*}
   and applying Lemma~\ref{lemma:error_est_after_integration}.

   Since $\uie$ solves \eqref{eq:diff_eq_u_ie} and $\laplace$ is linear,
   for the second estimate we note that
   \begin{align*}
     \laplace \left(\dd\right)^{-1} \uie &= \left(\dd\right)^{-1} \laplace \uie 
     =\left(\dd\right)^{-1} \dd \vie = \vie.
   \end{align*}
   and analogously $\laplace \left(\partial_t^{-1} \widetilde{u} \right)= \dot{\widetilde{u}}$.

   From the definition of $\varphi$ and the stability of the normal trace operator in $\HpLglobal{1}$
   we have 
   \begin{align*}
     \norm{\partial_t^{-1} \varphi(t_n) - \left[\left(\dd\right)^{-1} \varphi_{ie}\right]^n}_{H^{-1/2}(\Gamma)}
     &\leq C \norm{ \partial_t^{-1} \widetilde{u}(t_n) - \left[\left(\dd\right)^{-1} \uie\right]^n}_{\HpLglobal{1}} \\
     &\leq  C \left[\norm{ \partial_t^{-1} \widetilde{u}(t_n) - \left(\dd\right)^{-1} \uie}_{\Hpglobal{1}}  
       +\ltwonorm{\dot{\widetilde{ u}}(t_n) - \vie^{n}} \right].
   \end{align*}
   We apply Lemma~\ref{lemma:error_est_after_integration} twice to estimate the $H^1$ term of the integral
   by the $BL^1$ norm and the $L^2$ norm of the derivative up to higher order error terms. 
 \end{proof} 

\subsection{A scalar example} 
For the first example, we are only interested in convergence with respect to the time
discretization.
As geometry $\Omega^-$ we choose the unit sphere, the right-hand sides are chosen to be constant in 
space, i.e. $u^{inc}(x,t)=u^{inc}(t)$ and $\partial_n^+ u^{inc}(x,t)=0$ on $\Gamma$; see \cite{sauter_veit} for this approach in the linear case. The constant functions
are eigenfunctions of the integral operators $V(s),W(s),K(s),K^t(s)$ in the Laplace domain.
We therefore can replace the integral operators with $\lambda_{T}(s) M$ where for $T \in \{V,W,K,K^t\}$,
$\lambda_T(s)$ denotes the respective eigenvalue of $T(s)$ and $M$ is the mass matrix. It is also easy to see
that $\psi$ and $\varphi$ are constant in space and \eqref{eq:linearized_integral_eqn} can be reduced to a scalar problem which can be solved efficiently. 

\begin{example}
\label{ex:scalar_convergence}
Let $g(\mu):=\frac{1}{2}\mu+\abs{\mu}\mu$ and
  $u^{inc}(t)=-2\,e^{-10(t-t_0)^2}$ with
  $t_0=\pi/2$ and final time  $T=3$.
  The convergence of the method with respect to time 
  can be seen in Figure~\ref{fig:convergence_scalar}, where we plotted the error in approximating $\partial_t^{-1} \psi^n$.
  Since we only consider the scalar case, the error for $\varphi$ is virtually indistinguishable.
  We see that for both the implicit Euler and the BDF2 scheme, the full order of convergence is obtained.
  Investigating the solution, this is somewhat surprising, as the second derivative of $\psi$ has 
  a discontinuity; see Figure~\ref{fig:exact_psi_scalar}.
  Thus the BDF2 method performs better than predicted.
  Investigating the convergence of the Newton iteration, it appears on average that it is sufficient
  to make $3-4$ iterations to reduce the increment to $10^{-8}$, i.e., the additional cost due to the
  nonlinearity is negligible compared to the computation of the history.
\end{example}

\begin{figure}
  \centering

  \begin{subfigure}[b]{0.4\textwidth} 
    \begin{minipage}[t][10cm][t]{\textwidth}
      \includegraphics{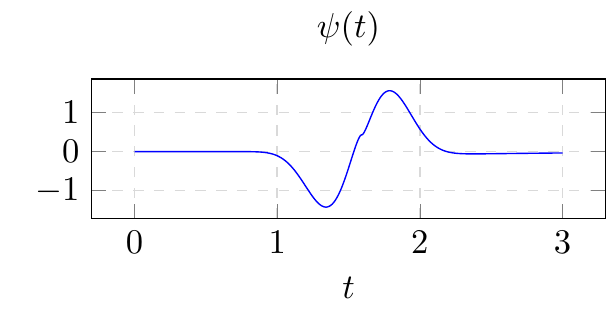}
      \includegraphics{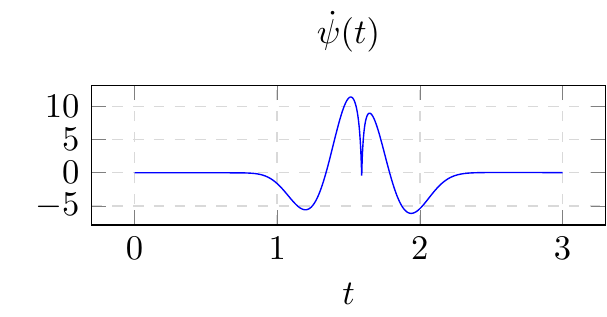}
      \includegraphics{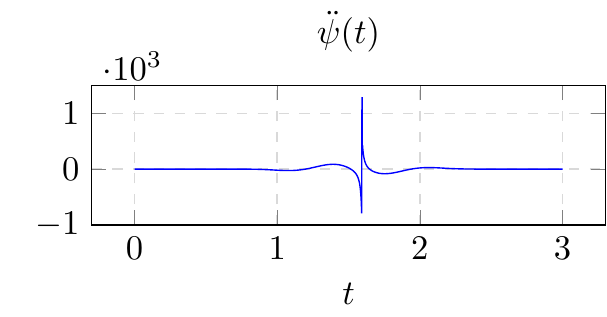}
    \end{minipage}
    \caption{Exact solution $\psi(t)$}
    \label{fig:exact_psi_scalar}
  \end{subfigure}
  \quad
  \begin{subfigure}[b]{0.5\textwidth}
    \begin{minipage}[t][10cm][t]{\textwidth}
     \includeTikzOrEps{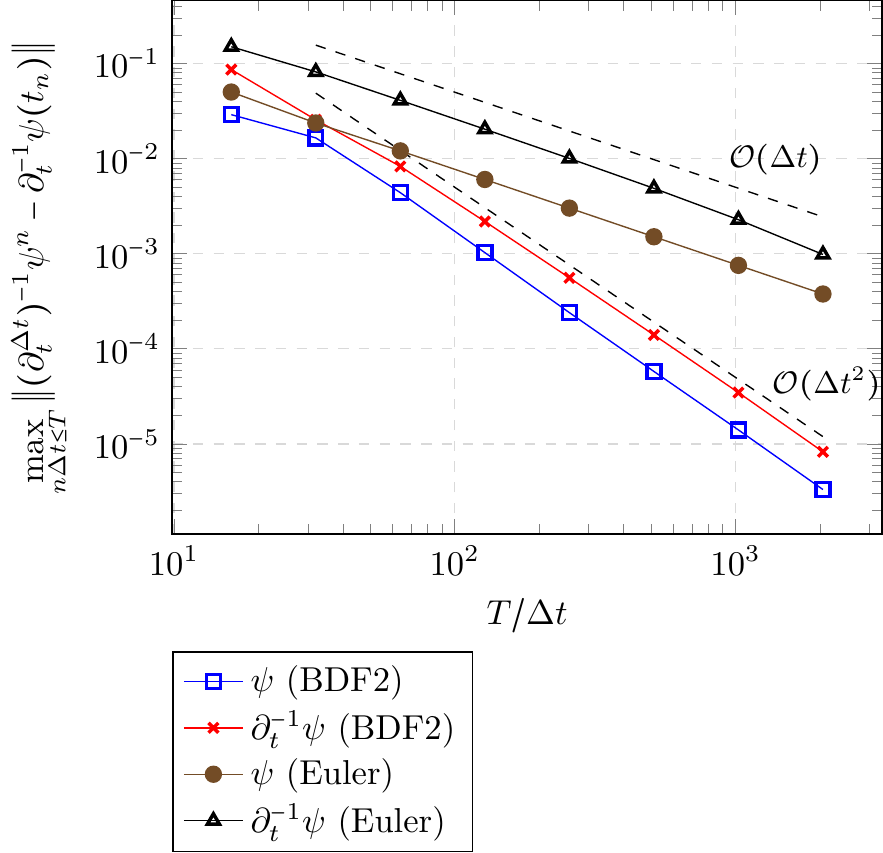}
   \end{minipage}
   \caption{Convergence rates }
   \label{fig:convergence_scalar}
  \end{subfigure}
  \caption{Numerical results for Example~\ref{ex:scalar_convergence} }
\end{figure}

\subsection{Scattering of a plane wave}
\begin{example}
\label{ex:fd_convergence}
In this example  $\Omega^-$ is the unit cube cube $[0,1]^3$ and $u^{inc}(x,t)$ is a
traveling wave given by
\begin{align*}
  u^{inc}(x,t):=e^{-A \left(x \cdot a - t - t_0 \right)^2},
\end{align*}
where $t_0=-2.5$, $A=8.0$, and $a=(1,-1,0)$.  We calculate the solution up to the end time
$T=4$ using a BDF2 scheme. 
For space discretization we use discontinuous piecewise linears for $X_h$ with $\operatorname{dim}(X_h)=7308$
and continuous piecewise quadratics with  $\operatorname{dim}(Y_h)=2438 $.
As an exact solution, we use the BDF2 approximation
with $\Delta t=\frac{4}{256}$ and the same spatial discretization.
Figure~\ref{fig:convergence_traveling_wave} shows that we can observe the optimal 
convergence rates of the method.
\end{example}

\begin{figure}
\centering
  \includeTikzOrEps{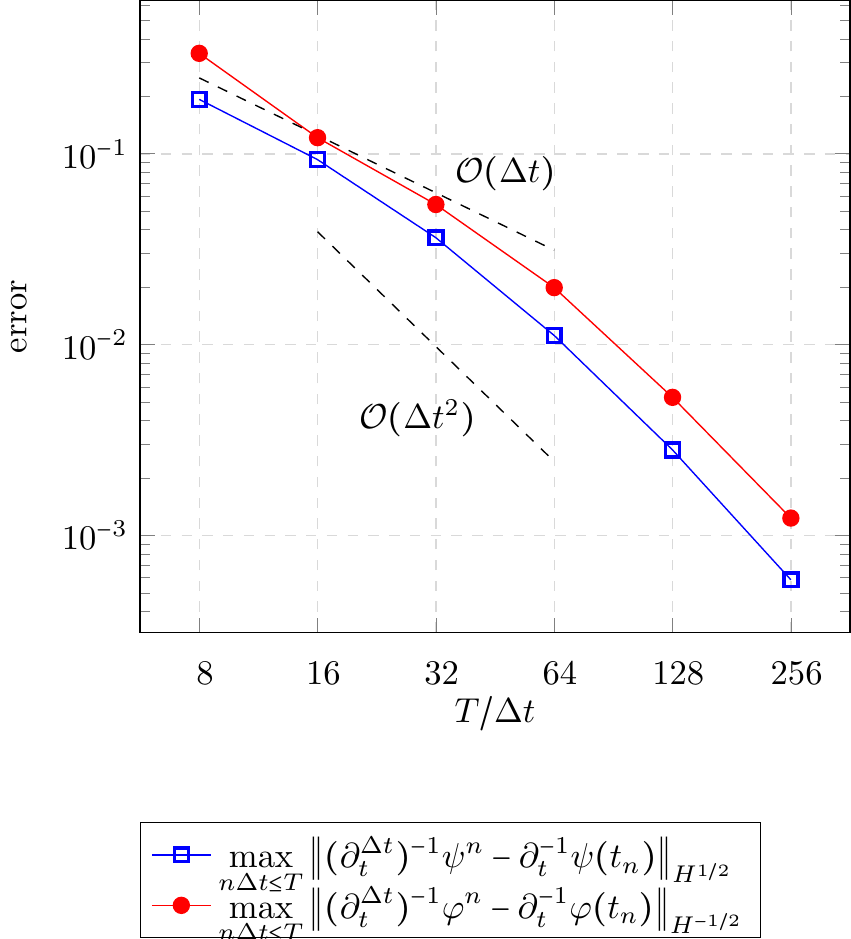}
\caption{Convergence for Example~\ref{ex:fd_convergence} }
\label{fig:convergence_traveling_wave}
\end{figure}

\subsection{Scattering from a nonconvex domain}
\label{subsect:nonconvex_scatter}
In Section~\ref{subsect:high_regularity}, we predicted optimal order of convergence, as long as the exact
solution is sufficiently smooth. This was the case of the numerical examples in
Examples~\ref{ex:scalar_convergence} and \ref{ex:fd_convergence}. In order to see whether this assumption
is indeed not always satisfied, we look at scattering from a more complex domain $\Omega^-$.

\begin{example}
\label{example:nonconvex}
We choose $\Omega^-$ the same as in \cite[Section 6.2.4]{banjai_algorithms}, as a body with a cavity
in which the wave can be trapped; see Figure~\ref{fig:exact_nonconvex}. We used $g(\mu):=\mu+\abs{\mu}\mu$
and $u^{inc}(x,t):=F(t - d\cdot x)$ with $F(s):=-\cos(\omega s) \, e^{- \left( \frac{s-A}{\sigma}\right)^2}$.
The parameters were $\omega:=\pi/2$, $\sigma=0.5$, $A=2.5$  and $d:=\sqrt{\frac{4}{5}}(1,0.5,0)^T$. For discretization we used a mesh 
of size $h \sim 0.05$ with discontinuous piecewise constants for $X_h$ and continous piecewise linears for $Y_h$,
  This gives $\operatorname{dim}(X_h)=15756$ and $\operatorname{dim}(Y_h)=7880$.
In Figure \ref{fig:convergence_nonconvex}, we see that we still get the full convergence rate $\bigO\left(\Delta t ^2\right)$.
\end{example}
\begin{figure}
\centering
  \includegraphics[width=0.5\textwidth]{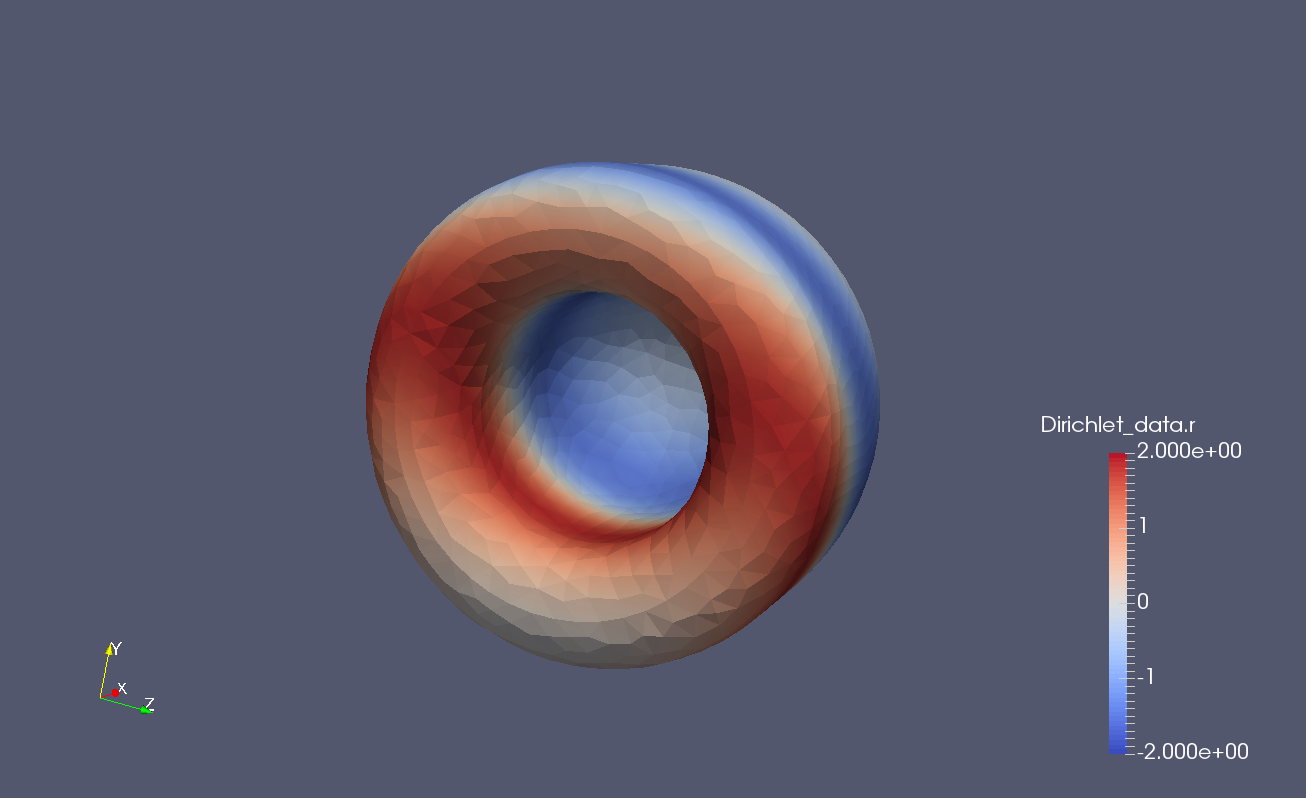}
  \caption{Exact solution to Example~\ref{example:nonconvex} for $t=2$}
  \label{fig:exact_nonconvex}
\end{figure}

\begin{figure}
\centering
\includeTikzOrEps{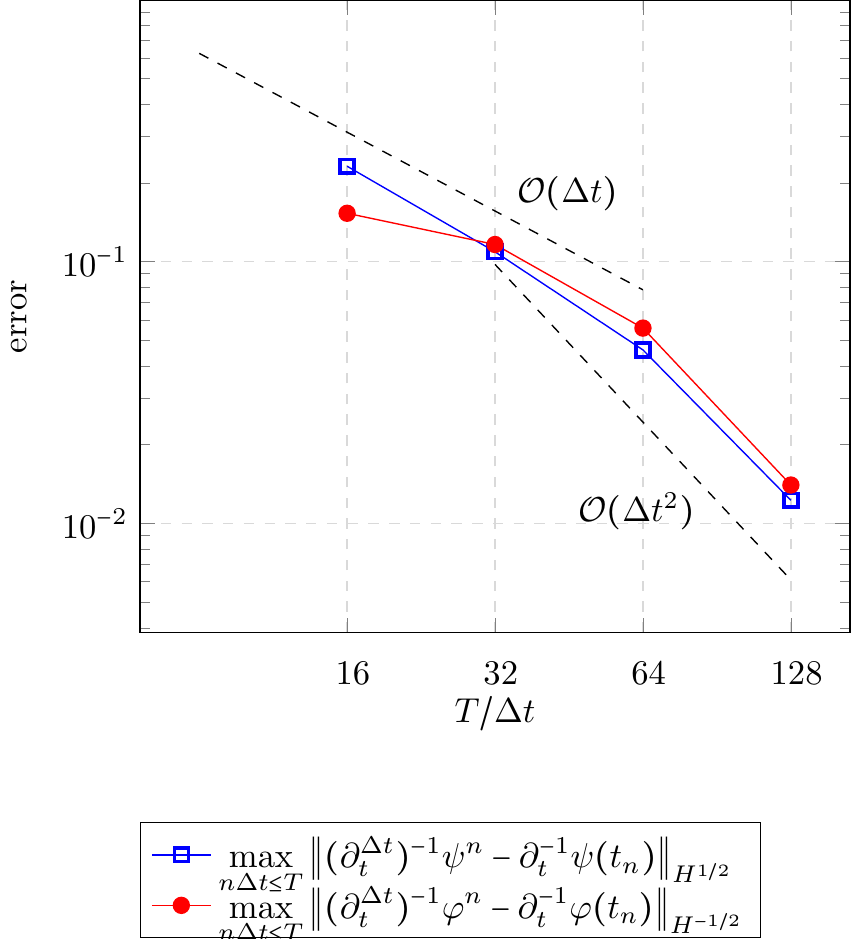}
  \caption{Convergence of the BDF2 method for Example~\ref{example:nonconvex}}
  \label{fig:convergence_nonconvex}
\end{figure}

Due to the computational effort involved, it is hard to say
 for which parameters a lowered convergence order manifests
and whether it is just due to some preasymptotic behavior. In order to better understand the behavior, we consider
the following model problem.
\begin{example}
\label{example:scalar_interior}
  We again consider the unit sphere, with an incoming wave that is constant in space.
  In order to construct a model problem, which is difficult for the numerical method we consider the extreme case
  of a ``completely trapping sphere'', i.e., the wave starts inside the sphere and has no way of escape, and investigate the convergence behaviour.
  This means we solve the interior boundary integral problem:
  \begin{align*}
    \begin{pmatrix}
      s V(s)  & -K - \frac{1}{2} \\
      \frac{1}{2}+K^t & s^{-1} W(s)    
      \end{pmatrix}
      \colvec{\varphi\\ \psi} + \colvec{0\\ g\left(\psi + \dot{u}^{inc}\right)}
              &= \colvec{0 \\ 0 }.
  \end{align*}
  We chose $g(\mu):=\frac{\mu}{4} + \mu\abs{\mu}$ and $u^{inc}(t):=\cos(\omega t) e^{-\left(\frac{t-A}{\sigma}\right)^2}$, with
  $A=2$, $\sigma=0.5$ and $\omega=4\pi$.
  In Figure~\ref{fig:convergence_scalar_interior}, we see that the BDF2 method no longer delivers the optimal convergence
  rate of $\Delta t^2$. For testing purposes, we also tried a convolution quadrature method based on the 2-step RadauIIA Runge-Kutta method which
  has classical order $3$ but only delivers second order convergence. Note, however, that even in the linear case Runge-Kutta based CQ  exhibits order reduction \cite{BanLM}.
\end{example}

\begin{figure}
  \centering
  \includeTikzOrEps{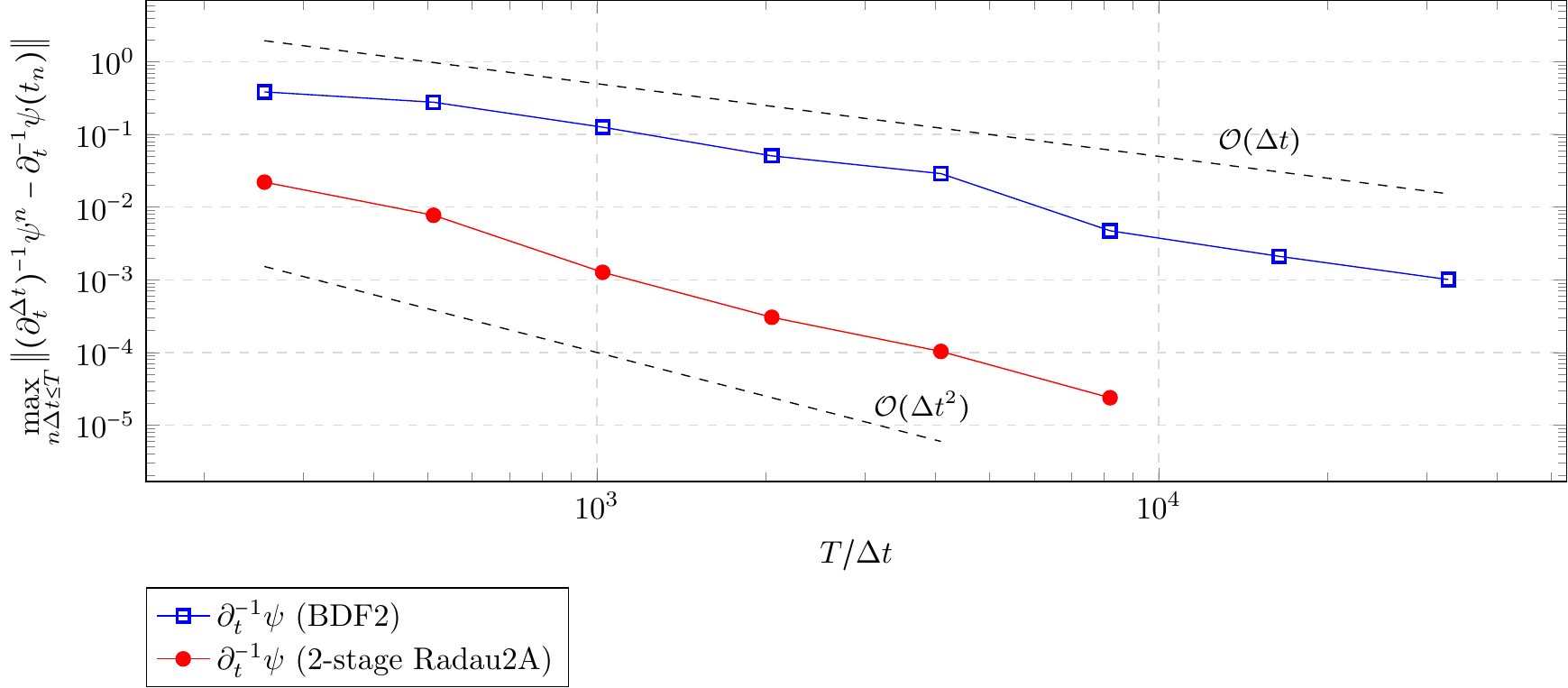}
  \caption{Convergence for Example~\ref{example:scalar_interior}}
  \label{fig:convergence_scalar_interior}
\end{figure}

\appendix
\section{Analysis based on integral equations and the Herglotz theorem} \label{appendix:herglotz}
\label{appendix:herglotz_analysis}
In this section, we would like to sketch a possible alternative approach to analysing the discretization scheme introduced in Problem~\ref{problem:fully_discrete_int_eq}. 
As these techniques require more regularity than the semigroup approach, we only consider the case of a smooth exact solution, analogous to Section~\ref{subsect:high_regularity}.

The basis of the approach is formed by the following two propositions, the first of which was proved in \cite{bls_fembem}.
\begin{proposition}
  \label{lemma:coercivity_calderon_td}
  For all $\sigma > 0$, $T>0$, there exists a constant $\beta > 0$, only dependent on $\Gamma$, $\sigma$ and $T$, such that
  \begin{align}
    \label{eq:corecivity_calderon_td}
    \int_{0}^{T}{ e^{-\sigma t}
    \dualproduct{B_{imp}(\partial_t) \colvec{\varphi \\ \psi }} {\colvec{\varphi \\ \psi}} \;dt}
 &\geq \beta \int_{0}^{T}{ e^{-\sigma t} \triplenorm{(\partial^{-1}_t \varphi, \partial_t^{-1} \psi)}^2 \;dt}
  \end{align}
  
  Analogously in the time discrete case, 
  for all $\sigma > 0$, $N \in \N$ with $T:=N \Delta t$ and $\rho:=e^{- \sigma T}$, there exists a constant $\beta > 0$, only dependent on $\Gamma$, $\sigma$ and $T$, such that
  \begin{align}
    \label{eq:corecivity_calderon_td_fd}
    \Delta t \sum_{n=0}^{N}{ \rho^{2n} \dualproduct{B_{imp}(\partial^{\Delta t}_t) \colvec{\varphi \\ \psi }} {\colvec{\varphi \\ \psi}} \;dt}
 &\geq \beta \Delta t \sum_{n=0}^{N}{ \rho^{2n} \triplenorm{\left((\dd)^{-1} \varphi, (\dd)^{-1} \psi\right)}^2 \;dt}.
  \end{align}
\end{proposition}

We also need a continuity result for the boundary operator $B(s)$.
\begin{proposition}
  There exists a constant $C>0$, depending only on $\Gamma$, such that
  \begin{align}
    \label{eq:continuity_calderon_ld}
    \triplenorm{B_{\text{imp}}(s) \colvec{\varphi \\ \psi }}
 &\leq C \abs{s}^2 \triplenorm{ \colvec{\varphi \\ \psi }}.
  \end{align}
\end{proposition}
\begin{proof}
  This inequality directly follows from the bounds shown in \cite{laliena_sayas}.
\end{proof}

In order to formulate the next result, we need the following Sobolev spaces for $r \geq 0$ and a Hilbert space $X$: 
\begin{align*}
  H^r_0((0,T),X):=\left\{g|_{(0,T)}: g \in H^r(\R,X) \text{ with } g \equiv 0 \text{ on } (-\infty,0) \right\}.
\end{align*}

We illustrate the approach and the difficulties it entails by focussing on the time discretization only.
\begin{theorem}
   Let $u$ solve \eqref{eq:ode_semigroup}, write $\widetilde{u}(t):=u(t) - u^{inc}(t)$ and define the traces $\psi(t):= \gamma^+ \dot{\widetilde{u}}(t)$ and
   $\varphi(t):=-\partial_n^+ \widetilde{u}(t)$.
   Let $X_h:=H^{-1/2}(\Gamma)$ and $Y_h:=H^{1/2}(\Gamma)$ be the full spaces and $(\varphi_{ie},\psi_{ie})$ the time-discrete numerical approximation 
   defined by~(\ref{eq:fully_discrete_int_eq}). Assume that $(\varphi,\psi) \in H_0^{r+1}\left((0,T), H^{1/2}(\Gamma)\times H^{-1/2}(\Gamma)\right)$, for $r > 5/2$.

   Then the following estimates holds
   \begin{align*}
     \sum_{n=0}^{N}{ \rho^{2n} \triplenorm{\colvec{\partial_t^{-1} \psi(t_n) - \left[\left(\dd\right)^{-1} \psi_{ie}\right]^n \vspace{1mm}\\ 
     \partial_t^{-1} \varphi(t_n) - \left[\left(\dd\right)^{-1} \varphi_{ie}\right]^n }}^2}    
     &\lesssim \left(\Delta t \right)^{2\min\left((r-2)\frac{p}{p+1},p\right)} \left( \norm{\varphi}^2_{H_0^{r+1}((0,T),H^{-1/2})} +  \norm{\psi}^2_{H_0^{r+1}((0,T),H^{1/2})} \right),
   \end{align*}
   for all $N$ with $N \Delta t \leq T$, with constants
   that depend only on $\Gamma$ and the end time $T$.
\end{theorem}
\begin{proof}
  Since we want to make use of the monotonicity of  of $g$, we start by shifting the Dirichlet traces  by $\dot{u}^{inc}$.
  Writing $\widetilde{\psi}:=\psi+ \dot{u}^{inc}$ and $\widetilde{\psi}_{ie}:=\psi_{ie}+ \dot{u}^{inc}$.
  We also  write $\Xi:=(\varphi,\psi+ \dot{u}^{inc})$ and $\Xi_{ie}:=(\varphi_{ie},\psi_{ie} + \dot{u}^{inc})$ for the pairs of functions and $E:=\Xi-\Xi_{ie}$ for the discretization error.
  We calculate, using the positivity property of Proposition~\ref{lemma:coercivity_calderon_td}:
  \begin{align*}
    \sum_{n=0}^{N}{\rho^{2n} \triplenorm{ \left[(\dd)^{-1} E \right]^n}^2}
    &\lesssim \sum_{n=0}^{N}{\rho^{2n} \dualproduct{B_{\text{imp}}(\dd) E}{E} 
      + \dualproduct{g(\widetilde{\psi}) - g(\widetilde{\psi}_{ie})}{\widetilde{\psi} - \widetilde{\psi}_{ie}}} \\ 
    &=\sum_{n=0}^{N}{\rho^{2n} \dualproduct{B_{\text{imp}}(\partial_t) \Xi - B_{\text{imp}}(\dd) \Xi_{ie}}{E} + \dualproduct{g(\widetilde{\psi}) - g(\widetilde{\psi}_{ie})}{\widetilde{\psi} - \widetilde{\psi}_{ie}} + \dualproduct{R_n(\Xi)}{E}} \\
  \end{align*}
  with the residual $R_n(\Xi):=\left[\left(B_{\text{imp}}(\partial_t) - B_{\text{imp}}(\dd) \right) \Xi\right](t_n) $.
  Since $\Xi$ and $\Xi_{ie}$ solve (\ref{eq:continuous_int_eq}) and (\ref{eq:fully_discrete_int_eq}) respectively, this gives the estimate:
  \begin{align*}
    \sum_{n=0}^{N}{\rho^{2n} \triplenorm{ \left[(\dd)^{-1} E \right]^n}^2}&\lesssim \sum_{n=0}^{N}{\rho^{2n}\dualproduct{R_n(\Xi)}{E}} 
                                                                            = \sum_{n=0}^{N}{\rho^{2n}\dualproduct{\dd R_n(\Xi)}{(\dd)^{-1}E}}
  \end{align*}
  (the second equality can be checked by using the Plancherel formula).
  The Cauchy-Schwarz inequality then gives the final estimate
  \begin{align*}
    \sum_{n=0}^{N}{\rho^{2n} \triplenorm{ \left[(\dd)^{-1} E \right]^n}^2}&\lesssim \sum_{n=0}^{N}{\rho^{2n} \triplenorm{\dd R_n(\Xi)}^2}.
  \end{align*}
  From the theory of convolution quadrature (\cite[Theorem 3.3]{lubich_94}) the residual can be estimated by:
  \begin{align*}
    \triplenorm{\dd R_n(\Xi)} &\lesssim  \triplenorm{ \dd \left[\left(B_{\text{imp}}(\partial_t) - B_{\text{imp}}(\dd) \right) \Xi\right]^n} \\
                              &\lesssim \triplenorm{ \left[\left(B_{\text{imp}}(\partial_t)\partial_t - B_{\text{imp}}(\dd) \dd \right) \Xi\right]^n}+
                                \triplenorm{  \left[\left(\partial_t - \dd \right)\left(B_{\text{imp}}(\partial_t) - B_{\text{imp}}(\dd) \right)\Xi\right]_n} \\
                              &\lesssim \left(\Delta t \right)^{\min\left((r-2)\frac{p}{p+1},p\right)} \left( \norm{\varphi}_{H_0^{r+1}((0,T),H^{-1/2})}
                                +  \norm{\widetilde{\psi}}_{H_0^{r+1}((0,T),H^{1/2})} \right),
  \end{align*}
  where in the last step we applied \cite[Theorem 3.3]{lubich_94} and absorbed the second contribution as $\partial_t$ is of a lower order of differentiation than $B(\partial_t)$, see (\ref{eq:continuity_calderon_ld}).
\end{proof}

\begin{remark} The fact that with this approach convergence can be proved more directly, at least for smooth solutions, without resorting to the results about the approximation of semigroups \cite{nevanlinna}, is an advantage. On the other hand, the requirements on the regularity of the exact solution is strictly larger than what is needed in Proposition~\ref{prop:semigroup_approx} ($p+4$ instead of $p+1$ continuous derivatives for full 
  convergence rates).  If we do not make any additional regularity assumptions, this new approach does not provide any predictions in regards to convergence. Additionally, the dependence on the end-time $T$ is much less clear (in general it will be some polynomial $\bigO(T^n)$, $n \in \N$). Similarly, when analysing the discretization error with regards to the spatial semidiscretization, 
  we also require at least one time derivative, due to the weak norm on the left hand-side that needs to be compensated by integration by parts. 
\end{remark}

\bibliographystyle{amsalpha}
\bibliography{bibliography}
\end{document}